\documentclass[12pt]{amsart}
\usepackage{setspace}
\usepackage{amssymb,color}
\usepackage{graphicx}
\usepackage{enumerate,hyperref}
\usepackage[
 margin=2.5 cm,
 includefoot,
 footskip=30pt,
]{geometry}

\usepackage{multirow}

\usepackage{tikz}
\usetikzlibrary{decorations.pathreplacing}
\newtheorem{thm}{Theorem}[section]

\newtheorem{lem}[thm]{Lemma}
\newtheorem{cor}[thm]{Corollary}

\newtheorem{conj}[thm]{Conjecture} 
\newtheorem{defn}[thm]{Definition}

\theoremstyle{definition}
\newtheorem{definition}[thm]{Definition}
\newtheorem{example}[thm]{Example}

\theoremstyle{remark}
\newtheorem{remark}[thm]{Remark}

\usepackage{enumerate}

\DeclareMathOperator{\less}{<_{\textrm{$\sigma$}}}
\DeclareMathOperator{\lessalt}{<_{alt}}

\DeclareMathOperator{\gessalt}{>_{alt}}
\DeclareMathOperator{\lesseq}{\le_{\textrm{$\sigma$}}}
\DeclareMathOperator{\gess}{>_{\textrm{$\sigma$}}}
\DeclareMathOperator{\gesseq}{\ge_{\textrm{$\sigma$}}}
\renewcommand{\S}{\mathcal{S}}
\DeclareMathOperator{\Int}{Int}
\newcommand{\C}{\mathcal{C}}
\DeclareMathOperator{\Al}{Allow}
\newcommand{\N}{\overline{N}}
\newcommand{\card}[1]{{\lvert #1 \rvert}} 	
\newcommand{\Wk}{\mathcal{W}_{k}}

\DeclareMathOperator{\asc}{asc}
\DeclareMathOperator{\des}{des}
\DeclareMathOperator{\Pat}{Pat}
\newcommand{\wmax}{\Omega_\sigma}
\newcommand{\wmin}{\omega_\sigma}

\newcommand{\multiset}[1]{\{\!\{#1\}\!\}}

\newcommand{\deleted}[1]{}

\newcounter{count}

\newcommand{\permutation}[1]{
\setcounter{indice}{0};
\foreach \i in {#1}
\addtocounter{indice}{1};
\addtocounter{indice}{1}
\draw [help lines] (0,0) grid (\theindice-1,\theindice-1);
\setcounter{indice}{1};
\foreach \i in { #1 } { 
\draw (\theindice-.5,\i-.5) [fill] circle (.18);
\addtocounter{indice}{1};
}
\addtocounter{indice}{-1};
}
\newcounter{indice}

\usetikzlibrary{shapes.geometric}

\usepackage{caption}

\title{Characterizations and Enumerations of Patterns of Signed Shifts}
\author{Sergi Elizalde}%
\email{sergi.elizade@dartmouth.edu}%
\thanks{The first author was partially supported by Simons Foundation grant \#280575.}
\author{Katherine Moore}%
\email{katherine.e.moore.gr@dartmouth.edu}%
\address{Department of Mathematics, Dartmouth College, Hanover, NH 03755, USA}%

\begin{document}

\maketitle

\begin{abstract}

Signed shifts are generalizations of the shift map in which, interpreted as a map from the unit interval to itself sending $x$ to the fractional part of $Nx$, some slopes are allowed to be negative.  Permutations realized by the relative order of the elements in the orbits of these maps have been studied recently by Amig\'o, Archer and Elizalde.   In this paper, we give a complete characterization of the permutations (also called patterns) realized by signed shifts. In the case of the negative shift, 
which is the signed shift having only negative slopes, we use our characterization to give an exact enumeration of these patterns. Finally, we improve the best known bounds for the number of patterns realized by the tent map, and calculate the topological entropy of signed shifts using these combinatorial methods.  
\end{abstract}

\section{Introduction}\label{sec:intro}

Permutations realized by one-dimensional dynamical systems give  insight into their short-term behavior and provide a powerful tool to distinguish random from deterministic time series~\cite{Amigobook}.  The close relationship between permutations and topological entropy, an important measure of complexity of a dynamical system, gives us an alternative derivation of this quantity.  

Given a linearly ordered set $X$, a map $f:X\to X$, and $x\in X$, consider the finite sequence $x,f(x),f(f(x)),\dots,f^{n-1}(x)$. If these $n$ values are different, then their relative order determines a permutation $\pi\in\S_n$ (we use $\S_n$ to denote the symmetric group on $\{1,2,\dots,n\}$), obtained by replacing the smallest value by a~$1$, the second smallest by a~$2$, and so on. 
We write $\Pat(x,f,n)=\pi$, and we say that $\pi$ is an {\em allowed pattern} of $f$, or that $\pi$ is {\em realized} by $f$, and also that $x$ {\em induces} $\pi$. For example, if $f(x) = \{ 3 x\}$, where $\{y\}$ denotes the fractional part of $y$ (see the left of Figure~\ref{fig:graphs} for a graph of this function),  and $x = .12$, we obtain $(x, f(x), f^2(x), f^3(x)) = (.12, .36, .08, .24)$,
and so $\Pat(f,x,4)= 2413$.  If there are repeated values in the first $n$ iterations of $f$ starting with $x$, then $\Pat(x,f,n)$ is not defined. Denote the set of allowed patterns by $\Al_n(f) = \{ \Pat(x, f, n) : x \in X\} \subseteq \S_n$ and $\Al(f) = \bigcup_{n \geq 1} \Al_n(f)$.  

It was shown in~\cite{BKP} that if $f$ is a piecewise monotone map on the unit interval, then the number of allowed patterns of length $n$ grows at most exponentially in $n$, implying the existence of forbidden patterns, that is, permutations that are not realized by $f$. Additionally, the logarithm of the growth rate of the number of allowed patterns equals the topological entropy of $f$. 

For a general piecewise monotone map, it is difficult to characterize and enumerate the set of allowed patterns.  Indeed, these questions have only been answered for functions that are variations on the shift.  In~\cite{EliShift}, a characterization and enumeration is given in the case that $f$ is a positive shift, that is, $f(x) = \{N x\}$ for some integer $N \geq 2$. Some progress when $f$ is a symmetric tent map has been made in~\cite{EliLiu}, and more recently in~\cite{KASLsymtent}. A characterization of allowed patterns when $f(x) = \{ \beta x\}$ for a real number $\beta>1$ was given in~\cite{EliBeta}. The case of negative $\beta$ has been addressed recently in~\cite{ECWSnegbeta} and~\cite{EMnegbeta}.  

An important class of dynamical systems are signed shifts, which generalize the shift, the negative shift, and the tent map. We will introduce them in Section~\ref{sec:signed}. A first approach to characterizing the allowed patterns of signed shifts appears in~\cite{AmigoSigned}, although it is cumbersome and incomplete.  An improvement is given in~\cite{KAchar}, yet the characterization provided there is not well suited for enumeration results. In this paper we provide a simple and concise characterization of the permutations realized by arbitrary signed shifts, which is given in Theorem~\ref{characterization} and proved in Section~\ref{sec:char}. Section~\ref{sec:intervals} describes, for any signed shift, the set of points inducing each pattern. In Section~\ref{sec:negativeshift} we give a formula for the smallest $N \ge 2$ such that a given pattern is realized by the negative shift $f(x) = \{-N x\}$. 
In Section~\ref{sec:enumeration} we provide an exact formula for the number of permutations realized by the negative shift,
which relies on our characterization of allowed patterns. In Section~\ref{sec:tent} we improve the best known bounds, given by Archer~\cite{KAchar},  for the number of patterns realized by the tent map.  Finally, in Section~\ref{sec:entropy} we provide an alternative derivation of the topological entropy of an arbitrary signed shift using permutations. The related problem of characterizing permutations realized by periodic orbits of signed shifts was studied in~\cite{KACyclicSigned}.

\section{Signed Shifts}\label{sec:signed}

Signed shifts are indexed by their {\em signature}, which is a $k$-tuple of signs $\sigma = \sigma_0 \sigma_1 \dots \sigma_{k-1} \in \{ + , - \}^k$, for $k\ge2$. Let $T_{\sigma}^+ = \{ t : \sigma_t = +\}$ and $T_{\sigma}^- = \{ t : \sigma_t = - \}$.  
Before defining the signed shift, which is a map on infinite words, we describe its counterpart as a map on the unit interval. Define the {\em signed sawtooth map} $M_{\sigma}: [0, 1] \rightarrow [0, 1]$, for each $0\le t\le k-1$ and $x \in [ \frac{t}{k}, \frac{t + 1}{k})$ (where the right endpoint of the interval is included when $t = k-1$), by letting
$$M_{\sigma}(x) = 
  \begin{cases}
kx - t & \text{if } t \in  T_{\sigma}^+, \\
t + 1 - kx & \text{if } t \in T_{\sigma}^-. 
\end{cases} $$
Examples of graphs of $M_\sigma(x)$ for some signatures $\sigma$ appear in Figure~\ref{fig:graphs}. 

\begin{figure}[htbp]
\centering
  \begin{tikzpicture}
      \begin{scope}[xshift = 0 cm, scale=.8] 
        \draw[->, thick](0, 0)--(0, 3.1875);
        \draw[->, thick](0, 0)--(3.1875, 0);
     	\draw[](0, 0)--(1, 3);
        \draw[](1, 0)--(2, 3);
        \draw[](2, 0) -- (3, 3);
      \end{scope}
	\begin{scope}[xshift = 3.5 cm, scale=.6]
        \draw[->, thick](0, 0)--(0, 4.25);
        \draw[->, thick](0, 0)--(4.25, 0);
     	\draw[](0, 4)--(1, 0);
        \draw[](1, 4)--(2, 0);
        \draw[](2, 4) -- (3, 0);
        \draw[](3, 4)--(4, 0);
      \end{scope}      
        \begin{scope}[xshift = 7 cm, scale=.6]
        \draw[->, thick](0, 0)--(0, 4.25);
        \draw[->, thick](0, 0)--(4.25, 0);
     	\draw[](0, 0)--(1, 4);
        \draw[](1, 4)--(2, 0);
        \draw[](2, 4) -- (3, 0);
        \draw[](3, 0)--(4, 4);
        \end{scope}
  \begin{scope}[xshift = 10.5 cm, scale=.6]
        \draw[->, thick](0, 0)--(0, 4.25);
        \draw[->, thick](0, 0)--(4.25, 0);
     	\draw[](0, 0)--(2, 4);
        \draw[](2, 4)--(4, 0);
      \end{scope}     
    \end{tikzpicture}
    \caption{The graphs of $M_\sigma$ for $\sigma = {+}{+}{+}$, $\sigma = {-}{-}{-}{-}$, $\sigma = {+}{-}{-}{+}$, and $\sigma = {+-}$, respectively. }
    \label{fig:graphs}
\end{figure}
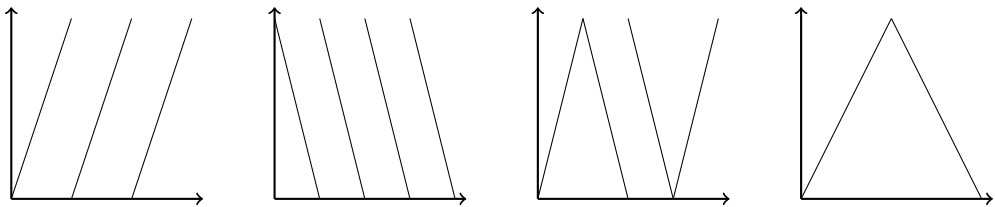

The signed shift $\Sigma_{\sigma}$, which is the main object of study in this paper, is closely related to $M_\sigma$, but defined on infinite $k$-ary words words instead of on real numbers. Let $\Wk$ be the set of infinite words on the alphabet $\{0, 1, \ldots, k{-}1\}$. We define a linear order $\less$ on $\Wk$ that depends on the signature $\sigma$.  

\begin{defn}\label{def:less} Let $\sigma \in \{ + , - \}^k$. For words $v, w \in \Wk$, define $$v_1 v_2 v_3 \ldots \less w_1w_2w_3 \ldots$$ if one of the following holds:
\begin{itemize}
\item $v_1 < w_1$, 
\item $v_1 = w_1 \in T_{\sigma}^+$ and $v_2 v_3 \ldots \less w_2 w_3 \ldots$, or 
\item $v_1 = w_1 \in T_{\sigma}^-$ and $v_2 v_3 \ldots \gess w_2 w_3 \ldots$.
\end{itemize}
Define the {\em signed shift} $\Sigma_{\sigma}: (\Wk, \less) \mapsto (\Wk, \less)$ by letting $\Sigma_{\sigma}(w_1 w_2 w_3 \ldots ) = w_2 w_3 \ldots$, where the linear order on $\Wk$ is $\less$.
\end{defn}

\begin{example} Let $\sigma = +-$ and $w = 1101^\infty \in \mathcal{W}_2$.  Prepending $1$ to the inequality $101^\infty \gess 01^\infty$,  we get $1101^\infty \less 101^\infty$.  Similarly, $ 01^\infty \less 1^\infty$ implies that $101^\infty \gess 1^\infty$ and $1101^\infty \less 1^\infty$.  Hence, 
$$01^\infty \less 1101^\infty \less 1^\infty \less 101^\infty,$$
that is,
$$ w_3 w_4 \ldots \less w_1 w_2 \ldots \less w_4 w_5 \ldots \less w_2 w_3 \ldots.$$
We conclude that $\Pat(w, \Sigma_{\sigma}, 4) = 2413$, and so $2413 \in \Al(\Sigma_{\sigma}).$
\end{example}

In this paper we focus on the signed shift $\Sigma_{\sigma}$, which is well suited to our combinatorial analysis. However, our results apply to the signed sawtooth map $M_\sigma$ as well, since these maps have the same allowed patterns.

\begin{lem} For any $\sigma \in \{ + , - \}^k$, we have $\Al(\Sigma_\sigma)=\Al(M_\sigma)$. 
\end{lem}

\begin{proof} Let
$$\mathcal{W}_k^0 =   \begin{cases}
\{ w \in \Wk : w \neq w_{[1, j-2]}0 \wmax \text{ for any } j \geq 3 \}& \text{if } \sigma_0 = \sigma_{k-1}= -,\\
 \{ w \in \Wk : w \neq w_{[1, j-1]} \wmax \text{ for any } j \geq 2\} & \text{otherwise},
\end{cases} $$
where $\wmax$ is the largest word in $\Wk$ under $\less$, given in Equation~\eqref{eq:wmax}.  It is shown in~\cite{KAchar} that $(M_{\sigma},<)$
is order-isomorphic to $(\mathcal{W}_{k}^0, \less)$, that is, there is an order-preserving bijection $\phi: [0, 1] \rightarrow \Wk^0$ such that  $\phi (M_{\sigma}(x)) = \Sigma_{\sigma}(\phi (x))$ for all $x$.  As a consequence, $M_\sigma$ has the same allowed patterns as the map $\Sigma_\sigma$ restricted to $\mathcal{W}_{k}^0$. Finally, we claim that the words in $\Wk \setminus \mathcal{W}_k^0$ do not contribute additional allowed patterns of $\Sigma_\sigma$, since they simply correspond to the distinct representations of points $x$ such that $f^{j-1}(x) = 1$.  

Indeed, since $M_{\sigma}(\frac{t}{k}) \in \{0, 1\}$, if a point $x$ satisfies $\Pat(x,M_{\sigma}, n)  = \pi$ and $M_{\sigma}^i(x)= \frac{t}{k}$ for some $i$, we must have $n \leq i + 2$.  In this case, one can take another point sufficiently close to $x$ that also induces $\pi$ under $M_{\sigma}$ and whose orbit does not contain an endpoint $\frac{t}{k}$.  Then, using the order-isomorphism $\phi$, we obtain a word in $\mathcal{W}_k^0$ inducing $\pi$.  We conclude that $\Al(\Sigma_{\sigma}) = \Al(M_{\sigma})$.    
\end{proof}

When $\sigma = +^k$ (we use this notation to denote $k$ copies of the $+$ sign), $\Sigma_\sigma$ is called the {\em $k$-shift} or {\em positive shift}, and the order $\less$ is the lexicographic order.  The signed shift with signature $\sigma = -^k$ is called the {\em $-k$-shift} or {\em negative shift}.  The shift with signature $\sigma = +-$ is the well-known tent map.   

\deleted{Since $M_\sigma$ and $\Sigma_{\sigma}$ are order-isomorphic except at the points of discontinuity of $M_\sigma$, and these points do not influence the realized permutations, we have $\Al(M_{\sigma}) = \Al(\Sigma_{\sigma})$.  Indeed, any pattern $\pi$, realized by a finite sequence $x, f(x), f^{2}(x), \ldots, f^{n-1}(x)$  that includes one of the endpoints $\frac{t}{k}$ for some $0 \leq t \leq k$ is also realized by some point $x'$ arbitrarily close to $x$, and $\Pat(x', M_{\sigma}, n) = \pi$. For our combinatorial analysis, it will be more suitable to work with the map $\Sigma_{\sigma}$.  }

Throughout this paper, we write $w=w_1w_2\dots$ and use the notation $w_{[i,j]}=w_iw_{i+1}\dots w_{j}$ and $w_{[i, \infty)} = w_i w_{i+1} \dots$.  If $d$ is a finite word, then $d^m$ denotes concatenation of $d$ with itself $m$ times, and $d^{\infty}$ denotes the corresponding infinite periodic word.  A finite word $d$ is \textit{primitive} if it cannot be written as a power of any proper subword, i.e. it is not of the form $d = a^m$ for any $m > 1$ and finite word $a$.  

\section{Characterization of Patterns of Signed Shifts}\label{sec:char}

In this section we give a characterization of the permutations realized by signed shifts. The first step is to apply a transformation on permutations that was introduced in~\cite{EliShift}.
 Let $\C_{n}^{\star}$ be the set of cyclic permutations of $\{1,2,\dots,n\}$ with a distinguished entry. We use the symbol ${\star}$ in place of the distinguished entry since its value can be recovered from the other entries, and we will frequently ignore the value replaced by the $\star$.  We will use both one-line notation and cycle notation while describing elements of $\C_{n}^{\star}$, which we call {\em marked cycles}. For example,
the cycle $(2,5, 1,4, 3)=45231$ with the entry 2 marked is denoted by $({\star}, 5, 1, 4, 3)=45{\star}31\in\C_5^{\star}$.

Consider the bijection from $\S_n$  to $\C_n^{\star}$ defined by $\pi \mapsto \hat{\pi}^{\star}$ where, if $\pi = \pi_1 \pi_2 \dots \pi_n$ in one-line notation, then $\hat{\pi}^{\star} = ({\star}, \pi_2, \dots, \pi_n)$ in cycle notation. Equivalently, $\hat{\pi}^{\star}=\hat{\pi}^{\star}_1\hat{\pi}^{\star}_2\ldots \hat{\pi}^{\star}_n$, where $\hat{\pi}^{\star}_j$ equals the entry to the right of $j$ in the one-line notation of $\pi$, and it equals $\star$ if there is no such entry. Thus, $\hat\pi^{\star}$ satisfies $\hat{\pi}^{\star}_{\pi_i} = \pi_{i +1}$ for $1 \leq i \le n-1$, and $\hat{\pi}^{\star}_{\pi_n}={\star}$.  

Intuitively, when constructing a word $w\in\Wk$ such that $\Pat(w,\Sigma_\sigma,n)=\pi$ for given $\pi$, the ability to use the same letter in two positions $i$ and $j$ of $w$ (that is, $w_i=w_j$) when $\pi_j = \pi_i + 1$ is determined by the relative order of the entries $\pi_{i+1}$ and $\pi_{j+1}$, which in turn is encoded by whether $\hat\pi$ has an ascent or descent at position $\pi_i$.  

For $1\le j\le n-1$, we say that $j$ is a {\em descent} of $\hat\pi$ if either $\hat{\pi}^{\star}_{j} > \hat{\pi}^{\star}_{j + 1}$, or $\hat{\pi}^{\star}_{j + 1}  = {\star}$ and $\hat{\pi}^{\star}_{j}  > \hat{\pi}^{\star}_{j +2}$.   Similarly, we say that a sequence $\hat\pi_i\hat\pi_{i+1}\dots\hat\pi_j$ is {\em decreasing} if the sequence obtained after deleting the $\star$, if applicable, is decreasing.  Ascents and increasing sequences are defined in the same fashion.  A variant of the following definition was first given in~\cite{KAchar}.  Throughout the paper, we assume that $\sigma = \sigma_0 \sigma_1 \dots \sigma_{k-1} \in \{ + , - \}^k$. Notice that a $\sigma$-segmentation is a slight modification of a decomposition of $\hat{\pi}^{\star}$ into ascending and descending blocks with monotonicity described by $\sigma$.  

\begin{definition}\label{def:segmentation} A \emph{$\sigma$-segmentation} of $\hat{\pi}^{\star}$  is a sequence of indices $E=(e_0,  e_1, \ldots, e_{k-1}, e_k)$ with $0=e_0\le e_1 \le \dots\le e_k=n$ satisfying the following conditions:
\begin{enumerate}[(a)]
\item for all $0\le t\le k-1$, the sequence $\hat{\pi}^{\star}_{e_t +1}\hat{\pi}^{\star}_{e_t +2} \dots \hat{\pi}^{\star}_{e_{t+1}}$ is increasing if $\sigma_{t} = +$ and decreasing if $\sigma_{t} = -$;
\item if $\sigma_0 = +$ and $\hat{\pi}^{\star}_1 \hat{\pi}^{\star}_2 = \star 1$ (equivalently, $\pi_{n-1}\pi_n = 21$), then $e_1 =0$; 
\item if $\sigma_{k-1} = +$ and $\hat{\pi}^{\star}_{n-1} \hat{\pi}^{\star}_n = n \star$ (equivalently, $\pi_{n-1}\pi_n = (n{-}1)n$), then $e_{k{-}1} = n-1$;  
\item if $\sigma_0 = \sigma_{k-1} = -$ and both $\hat{\pi}^{\star}_1 = n$ and $\hat{\pi}^{\star}_{n-1}\hat{\pi}^{\star}_{n} = 1{\star}$ (equivalently, $\pi_{n - 2} \pi_{n-1} \pi_n = (n-1) 1 n$), then either $e_1 =0$ or $e_{k{-}1} =  n-1$;
\item if $\sigma_0 = \sigma_{k-1} = -$ and both $\hat{\pi}^{\star}_1 \hat{\pi}^{\star}_2 = {\star}n$ and $\hat{\pi}^{\star}_n = 1$ (equivalently, $\pi_{n - 2} \pi_{n-1} \pi_n = 2 n 1$), then either $e_1 = 0 $ or $e_{k{-}1} = n$;
\item $e_t \neq \pi_n$ for all $1 \leq t \leq k-1$.  
\end{enumerate} 
To each $\sigma$-segmentation of $\hat{\pi}^{\star}$ we associate the finite word $\zeta = z_1 z_2 \dots z_{n-1}$  defined by $z_{i} = j$ whenever $e_j < \pi_i \leq e_{j+1}$, for $1\le i\le n-1$. We say that the $\sigma$-segmentation $E$ {\em defines}~$\zeta$, and that $\zeta$ is the {\em associated word} (or {\em prefix}) of the $\sigma$-segmentation.
\end{definition}

It will be convenient to visualize a $\sigma$-segmentation by placing vertical bars in positions $e_t$ (for $0\le t\le k$) in the one-line notation of $\hat\pi$, where position $0$ is considered to be to the left of $\hat\pi_1$, and positions $n$ is to the right of $\hat\pi_n$. 
We denote this visualization by $\hat\pi_E$. Condition~(f), which is equivalent to no bar being placed immediately after the $\star$ in $\hat\pi_E$ (except possibly the one for $e_k = n$), guarantees that each $\sigma$-segmentation of $\hat{\pi}^{\star}$ defines a distinct word $\zeta$. The number of entries other than $\star$ between adjacent bars in $\hat\pi_E$ determines the number of times each value appears in~$\zeta$.

\begin{example}\label{ex:seg} \begin{enumerate}[(a)] 
\item Let $\sigma = {+}{+}$ and $\pi = 52413$.  Then $\hat{\pi}^{\star} = 34{\star}12$ has a $\sigma$-segmentation $E= (0, 2, 5)$, which we visualize as $\hat\pi_E=|34|{\star}12|$. It defines the prefix $\zeta=1010$. Since $\pi_n = 3$, condition~(f) in Definition~\ref{def:segmentation} prevents us from choosing $(0, 3, 5)$, which would also have defined the same word $\zeta$.  
\item Let $\sigma = {+}{+}{-}$ and $\pi = 34521$.  Then $\hat{\pi}^{\star} = {\star}1452$ has a $\sigma$-segmentation $E= (0, 0, 3, 5)$, visualized as $\hat\pi_E=\|{\star}14|52|$, which defines $\zeta = 1221$.  Notice that $\hat{\pi}^{\star}$ does not have a $({+}{-})$-segmentation since, by condition~(b) in Definition~\ref{def:segmentation},  $\sigma_0= +$ and $\hat{\pi}^{\star}_1 \hat{\pi}^{\star}_2 = {\star}1$ force $e_1 = 0$.  
\item Let $\sigma = {-}{+}{+}$ and $\pi = 32145$.  Then $\hat{\pi}^{\star} = 4125{\star}$ has a $\sigma$-segmentation $E = (0, 2, 4, 5)$, visualized as $\hat{\pi}^{\star}_E = |41|25|{\star}|$, which defines $\zeta = 1001$.  Notice that $\hat{\pi}^{\star}$ does not have a $({-}{+})$-segmentation, since by condition~(c) in Definition~\ref{def:segmentation}, $\sigma_{1} = +$  and $\hat{\pi}^{\star}_{4} \hat{\pi}^{\star}_5 = 5 {\star}$ force $e_{2} = 4$.  
\item Let $\sigma = {-}{-}$ and $\pi = 2314$. 
Since $\sigma_0 = \sigma_{1} = -$ and $\pi_2 \pi_3 \pi_4 = 314$, condition~(d) in Definition~\ref{def:segmentation} forces either $e_1 = 0$ or $e_1 = 3$ in a $\sigma$-segmentation $(e_0, e_1, e_2)$ of $\hat{\pi}^{\star} = 431{\star}$. Taking $e_1 =0$, we get $E^+ = (0,  0, 4)$ and $\hat{\pi}^{\star}_{E^+} = ||431{\star}|$ , defining $\zeta^+ = 111$.  Taking $e_1 = 3$, we get $E^- = (0, 3, 4)$, $\hat{\pi}^{\star}_{E^-} = |431{\star}||$,  and $\zeta^- = 000$.
\item Let $\sigma = {-}{-}{-}$ and $\pi = 345261$. 
Since $\sigma_0 = \sigma_2 = -$ and $\pi_{4}\pi_{5}\pi_6 = 261$, condition~(e) in Definition~\ref{def:segmentation} forces either $e_1 = 0$ or $e_2 = 6$ in a $\sigma$-segmentation $(e_0, e_1, e_2, e_3)$ of $\hat{\pi}^{\star} = \star64521$. The two $\sigma$-segmentations are given by $E^{(1)} = (0, 0, 3, 6)$, which defines $\zeta^{(1)} = 12212$, and by $E^{(2)} = (0, 3, 6, 6)$, which defines $\zeta^{(2)}= 01101$.  Notice that, even though $\zeta^{(2)}$ is a binary word, $\hat{\pi}^{\star}$ does not have a $({-}{-})$-segmentation because of condition~(e).  
\end{enumerate}
\end{example}

Examples~\ref{ex:seg}(b)(c) illustrate a particular symmetry present in the patterns of signed shifts.   Let $E = (e_0, e_1, \ldots, e_k)$ be a $\sigma$-segmentation of $\hat{\pi}^{\star}$ defining the prefix $\zeta$.  Let $\rho = \pi^c$ be the complement of $\pi$.  Then $\hat\rho^\star = (\hat{\pi}^{\star})^{rc}$, where $(\hat{\pi}^{\star})^{rc}$ is the reverse-complement of $\hat{\pi}^{\star}$, i.e. $\hat{\rho}^{\star}_i = (n+1) - \hat{\pi}^{\star}_{n+1-i}$ for all $i$ except when $\hat{\pi}^{\star}_{n+1-i} = \star$, in which case $\hat{\rho}^{\star}_{i} = \star$.   Let $\sigma^r =\sigma_{k-1} \sigma_{k-2} \ldots \sigma_{0}$ be the reverse of $\sigma$.  Then a $\sigma^r$-segmentation of $\hat{\rho}^\star$ is given by $E' = (e'_0, e'_1, \ldots, e'_k)$, where $e'_i = n- e_{k-i}$ for all $i$, except if this would result in $e'_i = \rho_n$, in which case we let $e'_i = \rho_n - 1$ instead.  Moreover, the prefix $\zeta' = z'_{[1, n-1]}$ defined by $E'$ satisfies that $z'_i = k - z_i$ for all $1 \leq i \leq n-1$.  In many proofs in this paper we rely on this symmetry, by reducing two symmetric cases to one case. Note that in Definition~\ref{def:segmentation}, conditions (b) and (c) are related through this symmetry, and so are conditions (d) and~(e).  

\begin{example}\label{ex:sym}
Let $\pi  = 246135$, and consider the $(+-)$-segmentation of $\hat{\pi}^{\star}$ given by $E = (0, 4, 6)$ and $\hat{\pi}^{\star}_E = |3456|{\star}1|$, which defines $\zeta = 00100$.  Letting $\rho = \pi^c = 531642$, the above symmetry produces a $(-+)$-segmentation of $\hat\rho$ given by $E' = (0, 1, 6)$ and $\hat{\rho}^{\star}_{E'}= |6|{\star}1234|$, which defines $\zeta'= 11011$.
\end{example} 

Given a permutation $\pi\in\S_n$, a $\sigma$-segmentation of $\hat{\pi}^{\star}$, and its associated word $\zeta = z_{[1, n-1]}$, we define the following indices and subwords of $\zeta$. This notation will be used throughout the paper.
\begin{definition}\label{def:pq} If $\pi_n\neq n$, let $x$ be the index such that $\pi_x = \pi_n + 1$, and let $p = z_{[x,n-1]}$.  Similarly, if $\pi_n\neq 1$, let $y$ be such that $\pi_y = \pi_n - 1$, and let $q = z_{[y,n-1]}$.  
\end{definition}

For a finite word $d$ on the alphabet $\{0, 1, \ldots, k{-}1\}$, define $\|d\| = |\{i:  \sigma_{d_i} = -\}|$. 
For the $k$-shift, $\sigma=+^k$ and  $\|d\|$ is always zero; for the $-k$-shift, $\sigma=-^k$ and we have $\|d\| = |d|$, where $|d|$ denotes the length of $d$.  
By Definition~\ref{def:less},  
$$w \less v \text{ and }w_{[1, j]} = v_{[1, j]} \ \Longrightarrow \begin{cases}
w_{[j+1, \infty)} \less v_{[j+1, \infty)} & \text{if $\| w_{[1, j]} \|$ is even,}\\
w_{[j+1, \infty)} \gess v_{[j+1, \infty)} &  \text{if $\|w_{[1, j]}\|$ is odd.}\
\end{cases}$$

We will show that any word $w$ inducing $\pi$ has a certain form that may be described using $\sigma$-segmentations.  We show in Lemma~\ref{beginszeta} that, if $w$ induces $\pi$, there is a $\sigma$-segmentation of $\hat{\pi}^{\star}$ whose associated word is $\zeta = w_{[1, n-1]}$.  For this reason, we will refer to $\zeta$ as a \textit{prefix}.

\begin{definition}\label{def:invalid} A $\sigma$-segmentation of $\hat{\pi}^{\star}$ is  {\em invalid} if $\pi_n \notin \{1, n\}$ and the associated prefix $\zeta$ satisfies $p = q^2$ or $q = p^2$.  Otherwise the segmentation is {\em valid}. \end{definition}
 
Now we state the main theorem of this section, which gives the precise condition for when the existence of $\sigma$-segmentations implies that $\pi \in \Al(\Sigma_{\sigma})$. 

\begin{thm} \label{characterization}  Given a permutation $\pi$, we have $\pi \in \Al(\Sigma_{\sigma})$ if and only if there exists a valid $\sigma$-segmentation of $\hat{\pi}^{\star}$.  
\end{thm}

The above characterization of the allowed patterns of signed shifts is simpler than those given in~\cite{AmigoSigned} and~\cite{KAchar}, and it will allow us to obtain a complete description of the words $w \in \Wk$ inducing $\pi$ in Theorem~\ref{thm:intervals}, as well as enumeration results for the negative shift in Section~\ref{sec:enumeration}.
We note that the characterization given in~\cite{AmigoSigned} is incomplete. While this is corrected in \cite[Thm.~3.10]{KAchar}, the determination of whether
$\pi \in \Al(\Sigma_{\sigma})$ uses an unhandy condition, namely the requirement that no $b$ satisfies 
\begin{itemize}
\item $\pi_{n-b} < \pi_n < \pi_{n-2b}$ or $\pi_{n-2b} < \pi_n < \pi_{n-b}$, and 
\item $e_t < \pi_{n-b-i} \leq e_{t+1}$ if and only if $e_t < \pi_{n-i} \leq e_{t + 1}$ for all $1 \leq i \leq b$ and all $0 \leq t \leq k$.  
\end{itemize} 
The role of this condition will be played by our notion of an invalid $\sigma$-segmentation introduced in Definition~\ref{def:invalid}, which is significantly simpler.

The rest of this section will be devoted to proving Theorem~\ref{characterization}. Before launching into the proof, the following example, with diagrams included in Figure~\ref{fig:examples2}, illustrates how Theorem~\ref{characterization} can be used to determine  whether a permutation is an allowed pattern of a given signed shift.

\begin{example} \label{mainex}\begin{enumerate}[(a)]
 \item  Let $\sigma = ++$ and $\pi = 749862351$.  Then $\hat{\pi}^{\star} = {\star}35912468$ has a unique $\sigma$-segmentation $E = (0, 4, 9)$, and so $\hat{\pi}^{\star}_{E} = |{\star}359|12468|$, which defines the prefix $\zeta = 10111001$.  Since $\pi_n = 1$, this $\sigma$-segmentation is valid according to Definition~\ref{def:invalid}.  By Theorem~\ref{characterization}, $\pi$ is an allowed pattern of the $2$-shift. 

\item Let $\sigma = +-$ and $\pi = 356124$. Then $\hat{\pi}^{\star} = 245{\star}61$ has a $\sigma$-segmentation $E = (0, 3, 6)$, and so $\hat{\pi}^{\star}_{E} = |245|{\star}61|$.  This segmentation defines the prefix $\zeta =  01100$, which is valid because $p = 1100$ and $q = 01100$.  By Theorem~\ref{characterization}, $\pi$ is an allowed pattern of the tent map.  Another valid $\sigma$-segmentation of $\hat{\pi}^{\star}$ is $E' = (0, 5, 6)$, visualized as $\hat{\pi}^{\star}_{E'} = |245{\star}6|1|$, which defines the prefix $\zeta' = 00100$.  

\item Let $\sigma = --$ and $\pi = 615423$.  Then $\hat{\pi}^{\star} = 53{\star}241$ has a unique $\sigma$-segmentation $E= (0, 4, 6)$, and so $\hat{\pi}^{\star}_E = |53{\star}2|41|$, which defines the prefix $\zeta = 10100$.  Since $p = 00$ and $q = 0$, this $\sigma$-segmentation of $\hat{\pi}^{\star}$ is invalid because $p=q^2$.  To get a glimpse of the ideas behind the proof of Theorem~\ref{characterization}, let us see why there is no word $w = \zeta w_{[n, \infty)} \in \mathcal{W}_2$  inducing $\pi$. If $w$ were to induce $\pi$, then $w_{[y, \infty)} \less w_{[n, \infty)} \less w_{[x, \infty)}$, that is,
\begin{equation} \label{pqex} 0w_{[n, \infty)} \less w_{[n, \infty)} \less 00w_{[n, \infty)}, \end{equation}
which implies that $w_n = 0$.  By the definition of $\less$, canceling the common initial letter $0 \in T_{\sigma}^-$ in inequality~\eqref{pqex} implies $0w_{[n + 1, \infty)} \gess w_{[n + 1, \infty)} \gess 00 w_{[n + 1, \infty)},$ and so $w_{n+1} = 0$.  It follows from this argument that the only possibility is $w_{[n, \infty)} = 0^\infty$, which does not satisfy~\eqref{pqex}.  Since the only $\sigma$-segmentation  of $\hat{\pi}^{\star}$ is invalid, Theorem~\ref{characterization} implies that $\pi$ is not an allowed pattern of the $-2$-shift.  
\end{enumerate}
\end{example}

\begin{figure}[!htb]
 \centering
       \begin{tikzpicture}[scale= .33]
        \tikzstyle{scorestars}=[star, star points=5, star point ratio=2.25, draw,inner sep=1.3pt,anchor=outer point 3]
        \permutation{7, 3, 5, 9, 1, 2, 4, 6, 8}
        \draw [](0, 0) --  (4, 9);
        \draw [](4, 0) -- (9, 9);
            \draw[fill = white, white] (.4,6.5) circle (.27 cm);
                \draw (.25, 6.1) node[name=star, scorestars, fill=black]{};
        \draw[dashed](6.5, 3.5)--(3.5, 3.5);
        \draw[dashed](3.5, 3.5)--(3.5, 8.5);
        \draw[dashed](3.5, 8.5)--(8.5, 8.5);
        \draw[dashed](8.5, 8.5)--(8.5, 7.5);
        \draw[dashed](8.5, 7.5)--(7.5, 7.5);
        \draw[dashed](7.5, 7.5)--(7.5, 5.5);
        \draw[dashed](7.5, 5.5)--(5.5, 5.5);
        \draw[dashed](5.5, 5.5)--(5.5, 1.5);
        \draw[dashed](5.5, 1.5)--(1.5, 1.5);
        \draw[dashed](1.5, 1.5)--(1.5, 2.5);
        \draw[dashed](1.5, 2.5)--(2.5, 2.5);
        \draw[dashed](2.5, 2.5)--(2.5, 4.5);
        \draw[dashed](2.5, 4.5)--(4.5, 4.5);
        \draw[dashed](4.5, 4.5)--(4.5, .5);
        \draw[dashed](4.5, .5)--(.5, .5);
        \draw[dashed](.5, .5)--(.5, 6.5);
        \draw[dashed](.5, 6.5)--(6.5, 6.5);
        \draw[dashed](6.5, 6.5)--(6.5, 3.5);
        \draw[dotted](0, 0)--(9, 9);
      \end{tikzpicture}
      \qquad
            \begin{tikzpicture}[scale=.504] 
    \tikzstyle{scorestars}=[star, star points=5, star point ratio=2.25, draw,inner sep=1.3pt,anchor=outer point 3]
        \permutation{2, 4, 5, 3, 6, 1}
         \draw[fill = white, white] (3.5,2.5) circle (.4 cm);
        \draw[dashed](2.5, 4.5)--(4.5, 4.5);
        \draw[dashed](4.5, 4.5)--(4.5, 5.5);
        \draw[dashed](4.5, 5.5)--(5.5, 5.5);
        \draw[dashed](5.5, 5.5)--(5.5, .5);
        \draw[dashed](5.5, .5)--(.5, .5);
        \draw[dashed](.5, .5)--(.5, 1.5);
        \draw[dashed](.5, 1.5)--(1.5, 1.5);
        \draw[dashed](1.5, 1.5)--(1.5, 1.5);
     	\draw[dashed](1.5, 1.5)--(1.5, 3.5);
        \draw[dashed](1.5, 3.5)--(3.5, 3.5);
        \draw[dashed](3.5, 3.5) -- (3.5, 2.5);
        \draw[dashed](3.5, 2.5) -- (2.5, 2.5);
        \draw[dashed](2.5, 2.5) --(2.5, 4.5);
        \draw [](0, 0) --  (3, 6);
        \draw [](3, 6) -- (6, 0);
        \draw[dotted](0, 0)--(6, 6);
       
       \draw (3.34, 2.2) node[name=star, scorestars, fill=black]{};       
            \end{tikzpicture}
 \qquad
            \begin{tikzpicture}[scale=.504]
              \tikzstyle{scorestars}=[star, star points=5, star point ratio=2.25, draw,inner sep=1.3pt,anchor=outer point 3]
        \permutation{5, 3, 6, 2, 4, 1}
          \draw[fill = white, white] (2.5,5.5) circle (.3 cm);
        \draw[dashed](.5, 4.5)--(4.5, 4.5);
        \draw[dashed](4.5, 4.5)--(4.5, 3.5);
        \draw[dashed](4.5, 3.5)--(3.5, 3.5);
        \draw[dashed](3.5, 3.5)--(3.5, 1.5);
     	\draw[dashed](3.5, 1.5)--(1.5, 1.5);
     	\draw[dashed](1.5, 1.5)--(1.5, 1.5);
     	\draw[dashed](1.5, 1.5)--(1.5, 2.5);
     	\draw[dashed](1.5, 2.5)--(2.5, 2.5);
     	\draw[dashed](2.5, 2.5)--(2.5, 5.5);
     	 	\draw[dashed](2.5, 5.5)--(5.5, 5.5);
     	\draw[dashed](5.5, 5.5)--(5.5, .5);
     	\draw[dashed](5.5, .5)--(.5, .5);
     	\draw[dashed](.5, .5)--(.5, 4.5);
    
        \draw (0, 6) --  (4, 0);
        \draw (4, 6) -- (6, 0);
          \draw[dotted](0, 0)--(6, 6);
         
             \draw (2.34, 5.2) node[name=star, scorestars, fill=black]{};
      \end{tikzpicture}
\caption{Plots of $\hat\pi$ for $\pi = 749862351$, $\pi = 356124$ and $\pi = 615423$, from left to right, as in Example~\ref{mainex}.  The solid line segments with positive and negative slope illustrate a $\sigma$-segmentation in each case.}
    \label{fig:examples2}
 \end{figure}
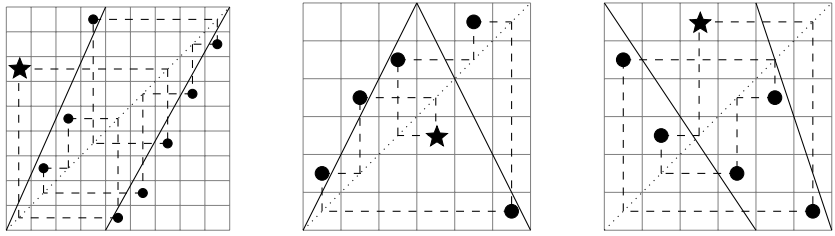

Theorem~\ref{characterization} follows from two main pieces. We first show in Lemmas~\ref{lem:cancomplete} and~\ref{beginszeta} that, if there is a word $w \in \Wk$ such that $\Pat(w, \Sigma_{\sigma}, n) = \pi$, then $\hat{\pi}^{\star}$ has a valid $\sigma$-segmentation defining the prefix $\zeta = w_{[1, n-1]}$.  Then, given a prefix $\zeta$ obtained from a valid $\sigma$-segmentation of $\hat{\pi}^{\star}$, we define words of the form $w = \zeta w_{[n, \infty)}$, and in Lemma~\ref{inducepi} we show that they induce $\pi$.  
Recall that $k$ always denotes the length of $\sigma$, that is, $\sigma \in\{+,-\}^k$. 
Lemmas~\ref{lem:cancomplete} and~\ref{beginszeta} are extended versions of \cite[Lem.~4.1,~4.2]{KAchar}, respectively.

\begin{lem} \label{lem:cancomplete} If the prefix $\zeta$ defined by a $\sigma$-segmentation of $\hat{\pi}^{\star}$ can be completed to a word $w=\zeta w_{[n,\infty)} \in \Wk$ with $\Pat(w, \Sigma_{\sigma}, n) = \pi$, then the $\sigma$-segmentation is valid. 
\end{lem}

\begin{proof} Suppose for contradiction that $\zeta = w_{[1, n-1]}$ is such that $p = q^2$ (the assumption $q =p^2$ leads to a contradiction via an analogous argument). Since $w$ induces $\pi$, we have $w_{[y, \infty)} \less w_{[n, \infty)} \less w_{[x, \infty)}$, or equivalently \begin{equation}\label{eq:q} q w_{[n, \infty)} \less w_{[n, \infty)} \less qq w_{[n, \infty)}.\end{equation}  If $\|q\|$ is even, $qw_{[n, \infty)} \less qq w_{[n, \infty)}$ implies that $w_{[n, \infty)} \less q w_{[n, \infty)} = w_{[y, \infty)}$, which is impossible because $w$ induces $\pi$ and $\pi_{y} = \pi_{n} - 1$.  If $\|q\|$ is odd, Equation~\eqref{eq:q} implies that $w_{[n, \infty)}$ begins with $q$, since it is between two words that begin with $q$. Writing $w_{[n, \infty)}=q w_{[2n-y, \infty)}$ and canceling the prefix $q$, we obtain $q w_{[2n-y, \infty)} \gess w_{[2n-y, \infty)} \gess qq w_{[2n-y, \infty)}$, which implies that $w_{[2n-y, \infty)}$ must start with $q$ as well. Repeating this argument, it follows that the only possibility would be $w_{[n, \infty)} = q^\infty$, but this choice of $w_{[n, \infty)}$ does not satisfy~\eqref{eq:q}.   Hence, the segmentation that produces $\zeta$ is valid.  \end{proof}

\begin{lem} \label{beginszeta} If $w\in\Wk$ and $\Pat(w, \Sigma_{\sigma}, n) = \pi$, then there exists a unique valid $\sigma$-segmentation of $\hat{\pi}^{\star}$ whose associated prefix is $\zeta = w_{[1, n-1]}$. 
\end{lem}

\begin{proof}  Let $w \in \Wk$ be such that $\Pat(w, \Sigma_{\sigma}, n ) = \pi$. For $0 \leq t \leq k$, let \begin{equation}\label{def:et} e_t = |\{ 1 \leq r \leq n : w_{r} < t\}|,\end{equation}
and let $e'_t = e_t$ unless $e_t=\pi_n$, in which case we let $e'_t = \pi_n -1$.
We will show that $(e'_0, e'_1, \dots, e'_k)$ is a valid $\sigma$-segmentation of $\hat{\pi}^{\star}$.  We note that the finite word $z_{[1, n-1]}$ defined by taking $z_i = j$ whenever $e_j < \pi_i \leq e_{j+1}$ for the set of indices $(e_0, e_1, \dots, e_k)$, as in Definition~\ref{def:segmentation}, is the same as that defined by $(e'_0, e'_1, \dots, e'_k)$, because the choice of $e'_t = \pi_n - 1$ only moves the index from the right to the left of $\hat{\pi}^{\star}_{\pi_n} = \star$.   

Since the entry $\hat{\pi}^{\star}_{\pi_n} = \star$ does not influence decreasing and increasing sequences in $\hat{\pi}^{\star}$, it suffices to show that condition~(a) in Definition~\ref{def:segmentation} holds for $(e_0, e_1, \dots, e_{k})$.  By Equation~\eqref{def:et}, the prefix $w_{[1, n]}$ has $e_t$ letters smaller than $t$.  Therefore, among the words $w_{[r, \infty)}$ with $1 \leq r \leq n$, there are exactly $e_t$ of them such that $w_r < t$, and exactly $e_{t+1}$ such that $w_r \le t$. Since $w$ induces $\pi$, it follows that if $e_t<\pi_i\le e_{t+1}$, then $w_{[i,\infty)}$ must be one of the shifts with $w_i \le t$ but not $w_i < t$, and so $w_i=t$. 

Consider first the case $t \in T_{\sigma}^-$.  To show that the sequence $\hat{\pi}^{\star}_{e_t +1}\hat{\pi}^{\star}_{e_t +2}\dots \hat{\pi}^{\star}_{e_{t+1}}$ is decreasing, suppose that $e_t < \pi_i < \pi_j \leq e_{t+1}$.  We will show that $\hat\pi_{\pi_i}>\hat\pi_{\pi_j}$ assuming that $i,j<n$, since the entry $\hat\pi_{\pi_n}=\star$ does not disrupt the property of $\hat{\pi}^{\star}_{e_t +1}\hat{\pi}^{\star}_{e_t +2}\dots \hat{\pi}^{\star}_{e_{k+1}}$ being decreasing. By the above paragraph, $w_i = w_j = t \in T_{\sigma}^-$, and $w_{[i, \infty)} \less w_{[j, \infty)}$ because $w$ induces $\pi$. Therefore, $w_{[i + 1, \infty)} \gess w_{[j + 1, \infty)}$, and so $\pi_{i + 1} > \pi_{j + 1}$, or equivalently $\hat{\pi}^{\star}_{\pi_i} = \pi_{i+1} > \pi_{j+1} = \hat{\pi}^{\star}_{\pi_j}$.

Likewise, to show that the sequence $\hat{\pi}^{\star}_{e_t+ 1} \hat{\pi}^{\star}_{e_t + 2} \dots \hat{\pi}^{\star}_{e_{t+1}}$ is increasing for all $t \in T_{\sigma}^+$, suppose that $e_t < \pi_i < \pi_j \leq e_{t +1}$.  Assuming that $i, j < n$, we will verify that $\hat{\pi}^{\star}_{\pi_i} > \hat{\pi}^{\star}_{\pi_j}$. As before, $w_i = w_j = t \in T_{\sigma}^+$ and $w_{[i, \infty)} \less w_{[j, \infty)}$.  Therefore, $w_{[i + 1, \infty)} \less w_{[j + 1, \infty)}$, and so $\hat{\pi}^{\star}_{\pi_i} = \pi_{i + 1} < \pi_{j +1} = \hat{\pi}^{\star}_{\pi_j}$.  

To show that condition (b) in Definition~\ref{def:segmentation} holds for $(e'_0, e'_1, \dots, e'_k)$, we will show that if $\sigma_0 = +$ and $\hat{\pi}^{\star}_{n - 1} \hat{\pi}^{\star}_n =  \star 1$ (equivalently, $\pi_{n-1} \pi_n =  2 1$) then $e_1 \leq 1$. Once this is proved, it will follow that $e'_1=0$, because even in the case that $e_1 = 1= \pi_n$, we would have $e'_1 = \pi_n -1 = 0$. Suppose for a contradiction that $e_1 > 1$.   Then, by Equation~\eqref{def:et}, the word $w_{[1, n]}$ has at least two $0$s.  Since $w$ induces $\pi$, we have that $w_{[n, \infty)}$ is the smallest word and $w_{[n - 1, \infty)}$ is the second smallest word among the shifts $w_{[r, \infty)}$ with $1 \leq r \leq n$.  It follows that $w_{n-1} = w_n=0$.  But then, canceling the initial $0$ of the inequality  $w_{[n, \infty)} \less 0 w_{[n, \infty)} = w_{[n - 1, \infty)}$ gives $w_{[n+1, \infty)} \less w_{[n, \infty)} = 0 w_{[n+1, \infty)}$, and so $w_{n+1}=0$. Repeating this process yields $w_{[n, \infty)} = 0^\infty = w_{[n - 1, \infty)}$, which is a contradiction because then $\Pat(w, \Sigma_{\sigma}, n)$ would not be defined.  Verifying condition~(c) for $(e'_0, e'_1,  \ldots, e'_n)$ follows a parallel argument.  

To show that condition~(d) holds, we verify that if $\sigma_0 = \sigma_{k-1}=-$,  $\hat{\pi}^{\star}_1 = n$ and $\hat{\pi}^{\star}_{n-1} \hat{\pi}^{\star}_{n} = 1 \star$ (equivalently, $\pi_{n-2} \pi_{n-1} \pi_{n} = (n{-}1)1n$), then either $e'_1 = 0$ or $e'_{k-1} = n-1$. 
It will be enough to show that $e_1=0$ and $e_{k-1} \geq n-1$. Indeed, the first equality clearly implies $e'_1=0$, and even in the case that $e_{k-1} = n = \pi_n$, we would still have $e'_{k-1} = \pi_n -1 = n-1$.  Suppose for a contradiction that $e_1 > 0$ and $e_{k-1} < n-1$. By Equation~\eqref{def:et}, the word $w_{[1, n]}$ has at least one $0$ and at least two $k-1$. 
Since $w$ induces $\pi$, we have that $w_{[n-1, \infty)}$ is the smallest and $w_{[n-2, \infty)}$ is the second largest among the shifts $w_{[r, \infty)}$ with $1 \leq r \leq n$. It follows that $w_{n-1} = 0$ and $w_{n-2} = k{-}1$.  
We claim that there is no choice for $w_{[n, \infty)}$ such that $\Pat(w, \Sigma_{-k}, n) = \pi$.  Indeed, since $\pi_n > \pi_{n-2}$, the  suffix $w_{[n, \infty)}$ must satisfy that
$$w_{[n, \infty)} \gess (k{-}1)0 w_{[n, \infty)} =w_{[n-2, \infty)}.$$
Therefore, $w_n= k{-}1$ and $w_{[n+1, \infty)} \less 0 w_{[n, \infty)}$.  From this, we conclude that $w_{n+1} = 0$ and $w_{[n+2, \infty)} \gess (k{-}1)0 w_{[n+2, \infty)} = w_{[n, \infty)}$.  Continuing inductively, the only possibility is that $w_{[n, \infty)} = ((k{-}1)0)^\infty$, which does not satisfy $w_{[n, \infty)} \gess w_{[n-2, \infty)}$.  Hence, condition~(d) in Definition~\ref{def:segmentation} holds for $(e'_0, e'_1, \dots, e'_k)$.  condition~(e) holds by a similar argument, and condition~(f) is immediate by construction.  

Finally, since the prefix $\zeta$ can be completed to a word $w$ inducing $\pi$, it follows from Lemma~\ref{lem:cancomplete} that the $\sigma$-segmentation $(e'_0, e'_1,  \ldots, e'_n)$ is valid.  Moreover, it is unique because each $\sigma$-segmentation of $\hat{\pi}^{\star}$ defines a distinct word $\zeta$.
\end{proof}

In the next three lemmas, let $\zeta$ be the prefix defined by some $\sigma$-segmentation of $\hat\pi^\star$.
Lemma~\ref{flipping} describes the relationship between $\zeta$ and $\pi$, while Lemmas~\ref{prefixlemsd2} and~\ref{lem:doubleback} give additional properties of~$\zeta$.

\begin{lem} \label{flipping} Let $i, j < n$. If $\pi_i < \pi_j$, then either \begin{enumerate}[(a)] \item $z_i < z_j$, or \item $z_i = z_j$ and $\pi_{j + 1} - \pi_{i + 1}$ has the same sign as $\sigma_{z_i}$. \end{enumerate}
\end{lem}

\begin{proof}  Suppose that $\pi_i < \pi_j$.  The construction of $\zeta$ in Definition~\ref{def:segmentation} yields $z_i \leq z_j$.  We will first prove that if $z_i = z_j$ and $\pi_{i + 1} < \pi_{j + 1}$, then $\sigma_{z_i} = +$.  By the definition of $\hat{\pi}^{\star}$, we have $\hat{\pi}^{\star}_{i} = \pi_{i + 1}$ and $\hat{\pi}^{\star}_j = \pi_{j + 1}$, and so $\hat{\pi}^{\star}_i < \hat{\pi}^{\star}_j$.  Moreover, using Definition~\ref{def:segmentation}, $z_i = z_j$ implies that there is no index $e_k$ in the $\sigma$-segmentation such that $\hat{\pi}^{\star}_i \leq e_k < \hat{\pi}^{\star}_j$.  Therefore, $z_i \in T_{\sigma}^+$ and so $\sigma_{z_i} = +$.  
A similar argument shows that if $z_i = z_j$ and $\pi_{i + 1} > \pi_{j + 1}$, then  $\sigma_{z_i} = -$.  
\end{proof}

\begin{lem}[{\cite[Lem.~4.5]{KAchar}}] \label{prefixlemsd2} Let $p$ and $q$ be as in Definition~\ref{def:pq}, when applicable. Then either $p$ is primitive, or $p = d^2$ for some primitive word $d$ with $\|d\|$ odd.  Likewise,  either $q$ is primitive or $q = d^2$ for some primitive word $d$ with $\|d\|$ odd.  \end{lem}

\begin{proof}  We will prove the first statement; the proof of the second statement follows by symmetry. We can write $p = d^r$, where $d$ primitive, and let $i = |d|$. Then $n = x+ri$ and $$d=z_{[x, x+i-1]} = z_{[x + i, x + 2i - 1]} = \ldots = z_{[x + (r-1)i, n-1]}.$$  

First suppose that $\|d\|$ is even. If $\pi_x < \pi_{x + i}$, then applying Lemma~\ref{flipping} $i$ times we obtain $\pi_{x+i} < \pi_{x + 2i}$.  Repeatedly applying this argument yields
$$\pi_x < \pi_{x+i} < \pi_{x + 2i} < \dots < \pi_{x + ri} = \pi_n,$$
which contradicts the fact that $\pi_x = \pi_n + 1$.  On the other hand, if $\pi_x > \pi_{x + i}$, then we obtain
$$\pi_x > \pi_{x+i} > \pi_{x + 2i} > \dots > \pi_{x + ri} = \pi_n.$$  Since $\pi_x = \pi_n + 1$, we must have $r = 1$, and so $p$ must be primitive in this case. 

Now suppose that $\|d\|$ is odd.  If $r$ is even, then we can write $p=(d')^{r/2}$ with $d' = d^2$ and apply the previous argument (which does not require $d'$ to be primitive) to conclude that $r/2=1$ and $p=d^2$, which is the second option in the statement. We are left we the case that $r$ is odd. 

If $\pi_x < \pi_{x + i}$, then Lemma~\ref{flipping} applied $i$ times implies that $\pi_{x + i} > \pi_{x + 2i}$.  Consider two cases depending on the relative order of $\pi_x$ and $\pi_{x+2i}$.
In the case $\pi_x < \pi_{x+2i} < \pi_{x + i}$, applying Lemma~\ref{flipping} $i$ times gives $\pi_{x+i} > \pi_{x + 3i} > \pi_{x + 2i}$.  Applying the same lemma $i$ more times we obtain $\pi_{x + 2i} < \pi_{x + 4i} < \pi_{x + 3i}$.  From repeated applications of this lemma, we get
$$\pi_x < \pi_{x + 2i} < \pi_{x + 4i} < \dots < \pi_{x + (r-1)i} < \pi_{x + ri} < \pi_{x + (r-2)i} < \dots < \pi_{x + 3i} < \pi_{x + i}.$$
Similarly, in the case that $\pi_{x + 2i} < \pi_x < \pi_{x + i}$, repeated applications of Lemma~\ref{flipping} give
$$\pi_{x + (r-1)i} < \dots < \pi_{x + 4i} < \pi_{x + 2i} < \pi_x < \pi_{x + i} < \pi_{x + 3i} < \dots < \pi_{x + ri}.$$
In both cases, we get $\pi_x < \pi_{x + ri} = \pi_n$, a contradiction to $\pi_x = \pi_n + 1$.  

If $\pi_x > \pi_{x + i}$, then Lemma~\ref{flipping} applied $i$ times implies that $\pi_{x + i} < \pi_{x + 2i}$.  Again, we consider two cases
depending on the relative order of $\pi_x$ and $\pi_{x+2i}$. If $\pi_{x + i} < \pi_x < \pi_{x +2i}$, then repeated applications of Lemma~\ref{flipping} give
$$\pi_{x + ri} < \dots < \pi_{x + 3i} < \pi_{x + i} < \pi_x < \pi_{x + 2i} < \pi_{ x + 4i} < \dots < \pi_{x + (r-1)i}.$$
Similarly, if $\pi_{x +i} < \pi_{x + 2i} < \pi_x$, then Lemma~\ref{flipping} gives
$$\pi_{x + i} < \pi_{x + 3i} < \dots < \pi_{x + ri} < \pi_{x + (r-1)i} < \dots < \pi_{x + 4i} < \pi_{x + 2i} < \pi_x.$$
In both cases, the fact that $\pi_x = \pi_n + 1= \pi_{x + ri}+1$ implies that $r = 1$, and so $p$ is primitive.\end{proof}

\begin{lem} \label{lem:doubleback} Let $p$ and $q$  be as in Definition~\ref{def:pq}, when applicable.  
If  $\zeta = a qq$ for some $a$ and $\|q\|$ is odd, then $p =q^2$.  Likewise, 
 if $\zeta = a' pp$ for some $a'$ and $\|p\|$ is odd, then $q = p^2$.
\end{lem}

\begin{proof}  Suppose that $\zeta = a qq$ and $\|q\|$ is odd; the case $\zeta = a' pp$ is proved similarly. Let $i = |q|$ and $m = n - 2i = y - i$, so that $z_{[m, y-1]} = z_{[y, n-1]} =q$, which is primitive by Lemma~\ref{prefixlemsd2} because $\|q\|$ is odd.  By the contrapositive of Lemma~\ref{flipping} applied $i$ times, $\pi_y < \pi_n$ implies that $\pi_{m} > \pi_y$.  Since $\pi_y=\pi_n-1$, we have $\pi_y < \pi_n <  \pi_m$.  We will show that $\pi_m = \pi_n +1$, from which it will follow that $m=x$ and $p = q^2$. Suppose for contradiction that there exists some $\pi_j$ such that $\pi_n < \pi_j < \pi_m$.

Consider first the case in which $j < y$. Since $\pi_y < \pi_j < \pi_m$, Lemma~\ref{flipping} applied $i$ times implies that $z_{[y, n-1]}=z_{[j, j+i - 1]}= z_{[m, y-1]} =q$ and $\pi_n=\pi_{y+i} > \pi_{j+i} > \pi_{m+i}=\pi_y$ since $\|q\|$ is odd, contradicting that $\pi_y = \pi_n - 1$.

Consider now the case in which $j > y$. 
Since $\pi_y < \pi_j < \pi_m$ and $z_{[y, m-1]} = z_{[m, n-1]} = q$, Lemma~\ref{flipping} applied $n-j$ times implies that $z_{[y, y + n-j - 1]} = z_{[j, n-1]} = z_{[m, m + n - j - 1]}$. If $\|z_{[j, n-1]}\|$ is odd, it yields $\pi_{y + n-j} > \pi_n > \pi_{m + n-j}$, and so $\pi_{y + n - j} > \pi_y > \pi_{m + n -j}$. Applying Lemma~\ref{flipping} $j-y$ more times, we get $z_{[y + n - j, n -1]} = z_{[y, j-1]} = z_{[m + n - j, y - 1]}$.  Similarly, if $\|z_{[j, n-1]}\|$ is even, we first get $\pi_{y + n - j} < \pi_n < \pi_{m + n - j}$ and so $\pi_{y + n - j} < \pi_y < \pi_{m + n -j}$, and Lemma~\ref{flipping} implies that $z_{[y + n - j, n -1]} = z_{[y, j-1]} = z_{[m + n - j, y - 1]}$ as well.
Combining the above equalities, we have $z_{[j, n-1]}z_{[y, j-1]}= z_{[y, y + n-j - 1]} z_{[y + n - j, n-1]}=z_{[y, n-1]}=q$, which states that $q$ is equal to one of its non-trivial cyclic shifts, thus contradicting that it is primitive. \end{proof}

It follows from Lemma~\ref{lem:doubleback} that if $\zeta$ is a prefix defined by an invalid $\sigma$-segmentation of $\hat{\pi}^{\star}$, then either $p = q^2$,  $q$ is primitive and $\|q\|$ is odd; or  $q = p^2$, $p$ is primitive and $\|p\|$ is odd.  In particular, when $\sigma = +^k$, all $\sigma$-segmentations are valid since $\|d\|$ is zero for any $d$.

Next we define sequences of words $s^{(m)}$ and $t^{(m)}$, which, as we will show, induce $\pi$ when $m \geq \frac{n}{2}$.  Denoting by $\wmax$ and $\wmin$ the largest and the smallest words in $\Wk$ with respect to $\less$, respectively, we have
\begin{equation}
\label{eq:wmax}
\wmax =   \begin{cases}
 (k{-}1)^\infty & \text{if } \sigma_{k{-}1} = +, \\
     (k{-}1)0^\infty & \text{if } \sigma_{k{-}1} = -, \sigma_0 = +,\\
     ((k{-}1)0)^\infty & \text{if } \sigma_{k{-}1} = -, \sigma_0 = -;
\end{cases}\quad
\wmin=   \begin{cases}
	0^\infty & \text{if } \sigma_0 = +,\\
     0(k{-}1)^\infty & \text{if } \sigma_0 = -,\sigma_{k{-}1} = +,\\
     (0(k{-}1))^\infty & \text{if }  \sigma_0 = -, \sigma_{k{-}1} = -.
\end{cases}
\end{equation}
When $\pi_n\neq n$ (so that $x$ and $p$ are defined) and $\zeta$ is defined by a valid $\sigma$-segmentation of $\hat{\pi}$, let 
$$ s^{(m)} =  \zeta p^{2m} \wmin.
$$
Similarly, when $\pi_n\neq 1$ (so that $y$ and $q$ are defined) and $\zeta$ is defined by a valid $\sigma$-segmentation of $\hat{\pi}$, let
$$ t^{(m)} =  \zeta q^{2m} \wmax. $$

The following result is a stronger version of ~\cite[Lem 4.6]{KAchar}. 

\begin{lem} \label{patdefined} If $\pi_n \neq n$ and $m \geq \frac{n}{2}$, then $\Pat(s^{(m)}, \Sigma_{\sigma}, n) $ is defined.  Likewise, if $\pi_n \neq 1$ and $m \geq \frac{n}{2}$, then $\Pat(t^{(m)}, \Sigma_{\sigma}, n)$ is defined.  
\end{lem}

\begin{proof} We prove the statement for $s^{(m)}$; the proof for $t^{(m)}$ is analogous.  By Lemma~\ref{prefixlemsd2}, the word $p$ is primitive or of the form $p=d^2$ for some primitive word $d$ with $\| d\|$ odd.  If $p$ is primitive, write $d = p$, so that in either case, $p^\infty = d^\infty$ with $d$ primitive.

For $\Pat(s^{(m)}, \Sigma_{\sigma}, n)$ to not be defined, we must have indices $1 \leq i < j \leq n-1$ such that $z_{[i, n-1]} p^{2m} \wmin = z_{[j, n-1]}p^{2m}\wmin$ or $z_{[i, n-1]}p^{2m} \wmin = p^{2m}\wmin$; noting that the first case reduces to the second upon canceling the prefix $z_{[i, i + n-j]} = z_{[j, n-1]}$ from both sides.  Since $z_{[i, n-1]}p^{2m} \wmin = p^{2m}\wmin$ and $2m \ge n$ implies that $(z_{[i, n-1]} p^{2m})_{[1, 2m(n-x)]} = p^{2m}$, we must have $z_{[i, n-1]} = d^r$ for some $r \ge 1$, where $d$ is primitive as above.  Thus, canceling equal prefixes $(z_{[i, n-1]} p^{2m})_{[1, 2m(n-x)]} = p^{2m}$ from $z_{[i, n-1]}p^{2m} \wmin = p^{2m}\wmin$, we now conclude that $d^r \wmin = \wmin$.   It follows that $\wmin$ must begin with $d^r$, and canceling prefixes and repeating the same argument we conclude that $\wmin = d^\infty = p^\infty$.  

Consider first the case when $\sigma_0 = +$, and so $\wmin = 0^\infty$. Since, by Lemma~\ref{prefixlemsd2}, $p$ is primitive or $p = d^2$ where $d$ is primitive and $\|d\|$ is odd, $\Pat(s^{(m)}, \Sigma_{\sigma}, n)$ would only be undefined in the case that $p = 0$.  Indeed, if we were to have $p = d^2$, then the only possibility is $d = 0$, which does not satisfy that $\|d\|$ is odd.   We claim that, for any permutation $\pi$, we have $p \neq 0$.  Suppose for a contradiction that $p = 0$, and note that $\pi_{x} = \pi_{n-1}$.  If there is an index $i < n-1$ such that $\pi_i < \pi_{x}$, take the maximal one.  By Lemma~\ref{flipping}, we have $z_i = z_x = 0$, and by Lemma~\ref{flipping}, we have $\pi_{i + 1} < \pi_{x + 1} = \pi_n < \pi_x$, a contradiction to the maximality of $i$.  We conclude that that there is no such index, and so $\pi_x = 2$ and $\pi_{n-1} \pi_n = 21$.   However, by Definition~\ref{def:segmentation}(b), we must have $e_1 =0$, and so $p = z_{n-1} = 0$ is impossible. Hence, when $\sigma_0 = +$, we have that $\Pat(s^{(m)}, \Sigma_\sigma, n)$ is always defined.

Consider the case when $\sigma_0 = -$ and $\sigma_{k-1} = +$.  We have  $\wmin = 0(k{-}1)^\infty$, which is not of the form $p^\infty$ for any $p$.  Thus, $\Pat(s^{(m)}, \Sigma_{\sigma}, n)$ is defined.  

Now consider the case when $\sigma_0 = \sigma_{k-1} = -$, and so $\wmin = (0(k{-}1))^\infty$.  For $\Pat(s^{(m)}, \Sigma_{\sigma}, n)$ to be undefined, we must have $p = 0(k{-}1)$.  Indeed, if we had $p = d^2$, the only possibility is $d = 0(k{-}1)$, which satisfies $\|d\| = 2$ and is not permitted by Lemma~\ref{prefixlemsd2}.  We claim that, for any permutation $\pi$, we have $p \neq 0(k{-}1)$. Note that $x = n - 2$ and so $\pi_{n- 2} = \pi_n + 1$ in this case.  If there is an index $i < n-2$ such that $\pi_i < \pi_{n-2}$, take the maximal one.  Since $z_{n-2} = 0$, Lemma~\ref{flipping} implies that $z_i = z_{n-2} = 0$ and $\pi_{i + 1} > \pi_{n-1}$.  Similarly, since $z_{n-1} = k{-}1$, applying Lemma~\ref{flipping} implies $z_{i+1} = z_{n-1} = k{-}1$ and $\pi_{i + 2} < \pi_n < \pi_{n-2}$, contradicting the maximality of $i$.  It follows that $\pi_{i} > \pi_{n-2}$ for all $i < n-2$.  We conclude that $\pi_{n-2} = 2$ and $\pi_n = 1$.  Moreover, if there were an index $j$ such that $\pi_j > \pi_{n-1}$, then Lemma~\ref{flipping} would give $z_{j} = z_{n-1} = k{-}1$ and $\pi_{j + 1} < \pi_{n} + 1 = 2$, which is impossible.  We conclude that $\pi_{n-1} = n$.  However, by Definition~\ref{def:segmentation}(d), in this case, a $\sigma$-segmentation of $\pi_{n-2} \pi_{n-1} \pi_n = 2n1$ has either $e_{k-1} = n$ or $e_1 = 0$.  If $e_{k-1} = n$, then $\zeta$ does not contain the letter $k{-}1$ and if $e_1 = 0$, then $\zeta$ does not contain the letter $0$.  Hence, in such a case, $p = 0(k{-}1)$ is impossible.  
\end{proof}

The proof of the following lemma is based on a similar statement given in \cite[Lem~4.7,~4.8]{KAchar}.  

\begin{lem}[\cite{KAchar}]  \label{nothingbetween}  Let $m \geq \frac{n}{2}$.  For the word $s = s^{(m)}$, we have $s_{[n, \infty)} \less s_{[x, \infty)}$ and there is no $1 \leq c \leq n$ such that $s_{[n, \infty)} \less s_{[c, \infty)} \less s_{[x, \infty)}$.   Likewise, for the word $t = t^{(m)}$, we have $t_{[y, \infty)} \less t_{[n, \infty)}$ and there is no $1 \leq c \leq n$ such that $t_{[y, \infty)}$ $\less t_{[c, \infty)} \less t_{[n, \infty)}$.  \end{lem}  

\begin{proof} 
We will prove the statement for $s$; the one for $t$ is analogous. The fact that $s_{[n, \infty)} = p^{2m} \wmin \less s_{[x, \infty)} = p^{2m+1 } \wmin$ follows because, after canceling equal prefixes, it is equivalent to  $\wmin \less p \wmin$, which holds because $\wmin$ is the smallest word in $\Wk$ with respect to $\less$ and the inequality must be strict because $\Pat(s^{(m)}, \Sigma_{\sigma}, n)$ is defined by Lemma~\ref{patdefined}.  

Next we prove that there is no $1 \leq c \leq n$ such that $s_{[n, \infty)} \less s_{[c, \infty)} \less s_{[x, \infty)}$, that is,
$$p^{2m} \wmin \less s_{[c, \infty)} \less p^{2m+1} \wmin.$$
Suppose for contradiction that such a $c$ existed. Then $s_{[c, \infty)} = p^{2m}v$ for some word $v$ satisfying $\wmin \less v \less p \wmin$.

We claim that $c<x$. If $p$ is primitive, this is because the first $p$ in $s_{[c, \infty)}$ cannot overlap with both the first and second occurrences of $p$ in $s_{[x, \infty)}$. If $p$ is not primitive, then by Lemma~\ref{prefixlemsd2}, $p = d^2$  where $d$ is primitive and $\|d\|$ is odd. The only way to have $c> x$ would be if $v = d\wmin$ is the largest word beginning with $d$, but this is impossible because $v \less p \wmin = d^2 \wmin$.

Next we show that $v$ begins with a $p$. Consider first the case when $p$ is primitive.  
Since $m \geq \frac{n}{2}$, one of the initial $n-2$ occurrences of $p$ in $s_{[c, \infty)}$ must coincide with the first occurrence of $p$ in $s_{[x, \infty)}$, since $|p^{n-2}|> n-1$, and so $v$ begins with $p$.  
If $p$ is not primitive, then $p = d^2$ where $d$ is primitive and $\|d\|$ is odd, by Lemma~\ref{prefixlemsd2}.  Since $|d^{2(n-2)}| > n -1$, one of the initial $2(n-2)$ occurrences of $d$ in $s_{[c, \infty)}$ must coincide with the first occurrence of $d$ in $s_{[x, \infty)}$.  Since $s_{[x, \infty)}$ begins with $d^{2(n-1)}$ and $c < x$, we have that $v$ begins with $d^2 = p$.

If $\|p\|$ is even, the fact that $v$ begins with a $p$ contradicts that $v \less p \wmin$, since $p \wmin$ is the smallest word beginning with $p$.  If $\|p\|$ is odd, then the above argument causes $\zeta$ to be of the form $\zeta = app$ for some $a$.  By Lemma~\ref{lem:doubleback}, this implies that $q=p^2$, contradicting the fact that $\zeta$ was obtained from a valid $\sigma$-segmentation. 

The proof for $t$ follows in a similar fashion.   \end{proof}

The statement of the following lemma appears in~\cite[Lem~4.9]{KAchar}. Here we give the first complete proof.

\begin{lem}\label{lem:patternsij} Let $w = \zeta w_{[n, \infty)} \in \Wk$ be such that $\Pat(w, \Sigma_{\sigma}, n)$ is defined.
If $w_{[x, \infty)} \gess w_{[n, \infty)}$ and there is no $1 \leq c \leq n$ such that $w_{[n, \infty)} \less w_{[c, \infty)} \less w_{[x, \infty)}$, then $\Pat(w, \Sigma_{\sigma}, n) = \pi$. Likewise, if $w_{[y, \infty)} \less w_{[n, \infty)}$ and there is no $1 \leq c \leq n$ such that $w_{[y, \infty)} \less w_{[c, \infty)} \less w_{[n, \infty)}$, then $\Pat(w,  \Sigma_{\sigma}, n) = \pi$.    \end{lem}

\begin{proof} We prove the assertion about $w_{[x, \infty)}$; the one about $w_{[y, \infty)}$ follows similarly. For $1 \leq i, j \leq n$, let $S(i, j)$ be the statement 
$$\pi_i < \pi_j \text{ implies } w_{[i, \infty)} \less w_{[j, \infty)}.$$
To show that  $\Pat(w,  \Sigma_{-N}, n) = \pi$, we will prove $S(i, j)$ for all $1 \leq i, j \leq n$ with $i \neq j$.  We consider three cases.  

\begin{enumerate}[(i)]
\item Case $i = n$.  Suppose that $\pi_n < \pi_j$.  By assumption, $w_{[n, \infty)} \less w_{[x, \infty)}$.  If $j = x$, we are done.  If $j \neq x$, then $\pi_n < \pi_j$  implies that $\pi_x < \pi_j$ since $\pi_x = \pi_n + 1$.  Thus, if $S(x, j)$ holds, then $w_{[n, \infty)} \less w_{[x, \infty)} \less w_{[j, \infty)}$, and so $S(n, j)$ must hold as well.  We have reduced $S(n, j)$ to $S(x, j)$.  Equivalently, $\neg S(n, j) \rightarrow \neg S(x, j)$, where $\neg$ denotes negation and $\rightarrow$ denotes implication.  

\item Case $j = n$.  Suppose that $\pi_i < \pi_n$.  In particular,  $i \neq n$ and  $\pi_i < \pi_x=\pi_n+1$. By assumption, in order to prove that $w_{[i, \infty)} \less w_{[n, \infty)}$, it is enough to show that $w_{[i, \infty)} \less w_{[x, \infty)}$.   Thus, we have reduced $S(i, n)$ to $S(i, x)$. 

\item Case $i, j < n$.  Suppose that $\pi_i < \pi_j$.  Let $m$ be such that $w_{[i, i + m -1]} = w_{[j, j + m -1]}$ and $w_{i + m} \neq w_{j + m}$.  First assume that $i + m, j + m \leq n -1$.  If $\|w_{[i, i + m -1]}\|$ is even, then Lemma~\ref{flipping} applied $m$ times to $\pi_{i} < \pi_{j}$ implies that $\pi_{i + m} < \pi_{j + m}$.  By Lemma~\ref{flipping}, we must have $w_{i + m} \leq w_{j + m}$, and we conclude that $w_{i + m} < w_{j + m}$.  Therefore, $w_{[i + m, \infty)} \less w_{[j + m, \infty)}$, and thus $w_{[i, \infty)} \less w_{[j, \infty)}$.
Similarly, if $\|w_{[i, i + m -1]}\|$ is odd, Lemma~\ref{flipping} applied $m$ times implies that $\pi_{i + m} > \pi_{j + m}$.  By Lemma~\ref{flipping}, we must have $w_{i + m} > w_{j + m}$ because $w_{i + m} \neq w_{j + m}$.  Therefore, $w_{[i + m, \infty)} \gess w_{[j + m, \infty)}$, and thus $w_{[i, \infty)} \less w_{[j, \infty)}$ again. This shows that if $i + m, j + m \leq n-1$, then $S(i,j)$ holds.

Suppose now that $i + m \geq n$ or $j+m \geq n$, and let $m'$ be the minimal index such that either $i + m' = n$ or $j + m' = n$.  Suppose first that $i + m' = n$ and $\|w_{[i, i  + m'-1]}\|$ is even. We claim that $S(i, j)$ reduces to $S(n, j + m')$ in this case. Indeed, suppose that $S(n, j + m')$, and let us show that $S(i,j)$ holds as well. If $\pi_i < \pi_j$, then Lemma~\ref{flipping} and the fact that $w_{[i, i + m' -1]} = w_{[j, j + m' - 1]}$ gives $\pi_n = \pi_{i + m'} < \pi_{j + m'}$.  Since $S(n, j +m')$ holds, we have $w_{[n, \infty)} \less w_{[j + m', \infty)}$, which implies $w_{[i, \infty)} \less w_{[j, \infty)}$, as desired.  Thus, $S(i, j)$ reduces to $S(n, j + m')$.  

Similarly, if $i + m' = n$ and $\|w_{[i, i + m'-1]}\|$ is odd, then $\pi_i < \pi_j$ and $w_{[i, i + m' - 1]} = w_{[j, j + m' - 1]}$ implies $\pi_{n} = \pi_{i + m'} > \pi_{j + m'}$ by Lemma~\ref{flipping}. If $S(j +m',n)$ holds, then $w_{[n, \infty)} \gess w_{[j + m', \infty)}$, which implies $w_{[i, \infty)} \less w_{[j, \infty)}$. Again, $S(i, j)$ reduces to $S(j + m', n)$ in this case.  

Now consider the case when $j + m' = n$ and $\|w_{[j, j + m'-1]}\|$ is odd. Then $\pi_i < \pi_j$ and $w_{[i, i + m' -1]} = w_{[j, j + m' - 1]}$ implies $\pi_{i + m'} > \pi_{j + m'} = \pi_n$ by Lemma~\ref{flipping}.  Therefore, $S(i, j)$ reduces to $S(n, i + m')$ in this case.  Finally, if $j + m' = n$ and $\|w_{[j, j + m'-1]}\|$ is even, $S(i, j)$ reduces to $S(i + m', n)$ by a similar argument.  
\end{enumerate}

In order to conclude that $S(i, j)$ holds for every $i,j$, we must show that the above process of reductions eventually terminates. Suppose for contradiction that the process goes on indefinitely.  Then at some point we would reach $S(x, l)$ with $l > x$, or $S(l, x)$ with $l > x$.    

\begin{itemize}
\item Suppose that we reach $S(x, l)$ with $l > x$. Since we assumed that the process does not terminate, case~(iii) above implies that $w_{[x, x+ n-l -1]} = w_{[l, n - 1]}$.
If $\|w_{[l, n-1]}\|$ is odd, then we get $\neg S(x,l) \rightarrow \neg S(n, x + n-l) \rightarrow \neg  S(x, x + n-l)$, using cases~(iii) and~(i). If $\|w_{[l, n-1]}\|$ is even, then $\neg S(x,l) \rightarrow \neg S(x + n-l,n) \rightarrow \neg  S(x + n-l,x)$, using cases~(iii) and~(ii). 
\item  Suppose we reach $S(l, x)$ with $l > x$.
Since the process does not terminate, case~(iii) implies that $w_{[l, n - 1]}=w_{[x, x+ n-l -1]}$.
If $\|w_{[l, n-1]}\|$ is odd, then we get $\neg S(l,x) \rightarrow \neg S(x + n-l,n) \rightarrow \neg  S(x + n-l,x)$. If $\|w_{[l, n-1]}\|$ is even, then $\neg S(l,x) \rightarrow \neg S(n,x + n-l) \rightarrow \neg  S(x,x + n-l)$.
\end{itemize}

In all cases, we conclude that $w_{[x, x+ n-l -1]} = w_{[l, n - 1]}$ and we reach $S(x,x + n-l)$ or $S(x + n-l,x)$. Now we can repeat the argument with $x+n-l$ playing the role of $l$ to deduce that $w_{[x, l -1]} = w_{[x+n-l, n - 1]}$ and obtain a reduction back to $S(x,l)$ or $S(l,x)$.

Combining the above equalities, we obtain $w_{[x, x + n - l - 1]} w_{[x + n -l - 1, n -1]} = w_{[x, l-1]} w_{[l, n-1]}  = p$.   Since $p$ is equal to some of its non-trivial cyclic shifts, it follows that $p$ is not primitive.  Thus, in the case that $p$ is primitive, we have verified the statements $S(x, l)$ and $S(l, x)$.

By Lemma~\ref{prefixlemsd2}, the case that remains is when $p = d^2$, where $d = w_{[x, l - 1]} = w_{[l, n-1]}$ and $\|d\|$ is odd.  Let $d = w_{[x, l -1]}$ and so we may write $p = d^2$.   It remains to verify the statements $S(x, l)$ and $S(l, x)$ in this case. 
 
To verify $S(l, x)$, suppose that $\pi_l < \pi_x$.  Since $\pi_n = \pi_x - 1$ and $l \neq n$, we have $\pi_{l} < \pi_n < \pi_x$.   We claim that, in this case, $\pi_{l} = \pi_y$.  
Suppose for contradiction that  $\pi_{l} \neq \pi_{n} - 1= \pi_y$.  Let $1 \leq h < n$  be the largest index such that $\pi_{l} < \pi_{h} < \pi_x$.

Consider first the case $h > l$. Since $w_{[l, n-1]} = w_{[x, l-1]}$, Lemma~\ref{flipping} applied $n-l$ times to $\pi_l < \pi_h < \pi_x$ implies \begin{equation}\label{eq:1}w_{[l, l + n - h - 1]} = w_{[h, n-1]} = w_{[x, x + n - h -1]}.\end{equation}   If $\|w_{[h, n-1]}\|$ is odd, then $n-h$ applications of Lemma~\ref{flipping} gives $\pi_{l + n - h} > \pi_{n} > \pi_{x + n - h}$, and so $\pi_{l + n -h} > \pi_x > \pi_{x + n - h}$ as well. Applying Lemma~\ref{flipping} $h-l$ more times, we conclude \begin{equation}\label{eq:2}w_{[l + n - h, n-1]} = w_{[x, x+ h - l - 1]}= w_{[l, h-1]},\end{equation} where in the last equality we used that $w_{[l, n-1]} = w_{[x, l-1]}$.
On the other hand, if $\|w_{[h, n-1]}\|$ is even, we get $\pi_{l + n - h} < \pi_{n} < \pi_{x + n - h}$ and $\pi_{l + n -h} < \pi_x < \pi_{x + n - h}$, from where Equation~\eqref{eq:2} holds as well. Combining Equations~\eqref{eq:1} and~\eqref{eq:2}, we get $d = w_{[l, n-1]} = w_{[l, h - 1]} w_{[h, n-1]} =  w_{[l + n - h, n -1]} w_{[l, l + n - h -1]}$, which states that $d$ is equal to one of its non-trivial cyclic shifts, thus contradicting that it is primitive.  

Now consider the case $h < l$.  Applying Lemma~\ref{flipping} $|d| = l - x$ times to the inequalities $\pi_l < \pi_{h} < \pi_x$, we obtain $\pi_n > \pi_{h + l - x} > \pi_{l}$ since $\|d\|$ is odd.  Therefore, $\pi_x > \pi_{h + l - x} > \pi_{l}$, which is a contradiction to the fact that we chose $h$ to be the largest index such that $\pi_{l} < \pi_{h} < \pi_x$.   

It follows that there is no index $h \neq n$ such that $\pi_{l} < \pi_{h} < \pi_x$.  We conclude that $\pi_{l} = \pi_y$, from which it follows that $d=q$ and $p = q^2$.  However, this contradicts that $\zeta$ comes from a valid $\sigma$-segmentation of $\hat{\pi}^{\star}$.
Since the assumption $\pi_l < \pi_x$ leads to a contradiction, the statement $S(l,x)$ trivially holds.

To verify $S(x, l)$, suppose now that $\pi_x < \pi_l$. We must show that $w_{[x, \infty)} \less w_{[l, \infty)}$.  Suppose to the contrary that $w_{[l, \infty)} \less w_{[x, \infty)}$.  Then, by assumption, $w_{[l, \infty)} \less w_{[n, \infty)} \less w_{[x, \infty)}$.  Hence,
\begin{equation}\label{eq:ineqd}
d^2 w_{[n, \infty)} \less w_{[n, \infty)} \less d w_{[n, \infty)}.
\end{equation}
Therefore, $w_{[n, \infty)}$ must begin with $d$, and by canceling equal prefixes and repeating this process,  we determine that the only option would be $w_{[n, \infty)} = d^\infty$, which does not satisfy the inequalities~\eqref{eq:ineqd}.  We conclude that $w_{[x, \infty)} \less w_{[l, \infty)}$, and so $S(x, l)$ holds.

We have shown that if the above process of reductions does not terminate, then it reaches $S(x,l)$ or $S(l,x)$ where $l-x=n-l$ and $p=d^2$. We have now shown that both $S(x,l)$ and $S(l,x)$ hold in this case. It follows that $S(i, j)$ holds for all $1 \leq i, j \leq n$ with $i \neq j$.  
\end{proof}

\begin{lem} \label{inducepi} If $\pi_n \neq n$ and $m \geq \frac{n}{2}$, then $\Pat(s^{(m)}, \Sigma_{\sigma}, n) = \pi$.  Likewise, if $\pi_n \neq 1$ and $m \geq \frac{n}{2}$, then $\Pat(t^{(m)}, \Sigma_{\sigma}, n) = \pi$.  \end{lem} 

\begin{proof} By Lemma~\ref{patdefined}, we have that $\Pat(s^{(m)}, \Sigma_{\sigma}, n)$ is defined. Moreover, using Lemma~\ref{nothingbetween} and Lemma~\ref{lem:patternsij}, we conclude that $s^{(m)}$ induces $\pi$.  The statement for $t^{(m)}$ can be proved similarly. 
\end{proof}

\begin{proof}[Proof of Theorem~\ref{characterization}] We first showed in Lemma~\ref{lem:cancomplete} that if there is a word $w \in \Wk$ such that $\Pat(w, \Sigma_{\sigma}, n) = \pi$, then $\hat{\pi}^{\star}$ has a valid $\sigma$-segmentation.  On the other hand, given a valid $\sigma$-segmentation of $\hat{\pi}^{\star}$, we produced words $s^{(m)}$ or $t^{(m)}$ (at least one of which always exists) and showed in Lemma~\ref{inducepi} that they induce $\pi$.  
\end{proof}

\begin{cor}[\cite{KAchar}]  \label{cor:containment} If $\sigma$ contains $\tau$ as a (not necessarily consecutive) subsequence, then  
$$\Al(\Sigma_{\tau}) \subseteq \Al(\Sigma_{\sigma}).$$
\end{cor}
\begin{proof} Write $\tau = \tau_0 \tau_1 \ldots \tau_{k-1}$ and $\sigma = \sigma_0 \sigma_1 \ldots \sigma_{k'-1}$.  
Let $(e_0, e_1, \ldots, e_k)$ be a valid $\tau$-segmentation of $\hat{\pi}^{\star}$, which exists by Theorem~\ref{characterization}. Let $0=  i_0 < i_1 < \ldots i_{k-1} < i_k = k'$ be a choice of indices $i_j$ such that $\sigma_{i_j} = \tau_j$ for $0 \leq j < k$.   Define the indices $(e'_0, e'_1, \ldots, e'_{k'})$ by taking $e'_0 = 0$ and $e'_t = e_{j}$ if $i_{j-1} < t \leq i_{j}$ for $1 \leq t \leq k'$.  To see why condition~(a) of Definition~\ref{def:segmentation} holds, notice that the sequence $\hat{\pi}^{\star}_{e'_t + 1} \hat{\pi}^{\star}_{e'_t+2}\ldots \hat{\pi}^{\star}_{e'_{t+1}}$ is empty whenever $t \neq i_j$ for some $0 \leq j < k$. In the case that $t = i_j$, we have $e'_t= e_j$ and $e'_{t+1} = e_{j+1}$. It follows that $\hat{\pi}^{\star}_{e'_{t} + 1} \hat{\pi}^{\star}_{e'_{t} +2} \ldots \hat{\pi}^{\star}_{e'_{t+1}} = \hat{\pi}^{\star}_{e_j + 1} \hat{\pi}^{\star}_{e_{j}} \ldots \hat{\pi}^{\star}_{e_{j+1}}$ is increasing if $\sigma_{i_j} = \tau_j = +$, and it is decreasing if $\sigma_{i_j} = \tau_j = -$.     The remaining conditions of Definition~\ref{def:segmentation} can be proved with similar arguments.    Moreover, the $\sigma$-segmentation is valid since we reassigned letters such that $z_t = j$ if and only if $z'_t = i_j$, and so $\|z_t\| = \|z'_t\|$ for all $1 \leq t < n$. Therefore, $(e_0', e_1', \ldots, e_k')$ is a valid $\sigma$-segmentation of $\hat{\pi}^{\star}$ and by Theorem~\ref{characterization}, we have $\pi \in \Al(\Sigma_{\sigma})$.
\end{proof}

\begin{example} Let $\tau = ++$ and $\sigma = {+}{-}{+}{+}$ be signatures of signed shifts.  Take $\pi = 3741526 \in \Al(\Sigma_{\tau})$, and so $\hat{\pi}^{\star} = 56712{\star}4$. We take the $\tau$-segmentation of $\hat{\pi}^{\star}$ given by $E = (0, 3, 7)$, and so $\hat{\pi}^{\star}_{E} = |567|12{\star}4|$,  defining $\zeta_{\tau} = 011010$.   Removing $\sigma_1$ and $\sigma_3$ from $\sigma$ leaves $\tau$, so we can take $E'= (0, 3, 3, 7, 7)$ as our $\sigma$-segmentation, and so $\hat{\pi}^{\star}_{E'} = |567\|12{\star}4\|$. 
This valid segmentation defines the prefix $\zeta_{\sigma} = 022020$, and we conclude that $\pi \in \Al(\Sigma_{\sigma})$.  
\end{example}

\section{Allowed Intervals}\label{sec:intervals}

Fix a signed shift $\Sigma_{\sigma}$ with signature $\sigma = \sigma_0 \sigma_1 \ldots \sigma_{k-1}$, and fix $n$.
In this section, for each $\pi \in \Al_n(\Sigma_{\sigma})$, we give a complete description of the set of words $w \in \Wk$ inducing $\pi$. We denote this set by
$$\Int(\pi, \Sigma_{\sigma}) = \{ w \in \Wk : \Pat(w, \Sigma_{\sigma}, n) = \pi \}.$$
In Theorem~\ref{thm:intervals}, we show that $\Int(\pi, \Sigma_{\sigma})$ is a finite union of disjoint intervals. An open interval in $\Wk$ is defined as a set of the form $\{w \in \Wk : u \less w \less v\}$ for some $u, v \in \Wk$, and denoted by $(u, v)_{\less}$. Intervals $[u, v)_{\less}$, $(u, v]_{\less}$ and $[u, v]_{\less}$ are defined similarly.  
We then use these intervals to give an upper bound on the number of allowed patterns of $\Sigma_\sigma$, which later in Section~\ref{sec:entropy} will be used to calculate the topological entropy of this map.  The following technical lemma will be useful later in this section.

\begin{lem} \label{dinfinity} Let $d$ be a finite word, and $v$ an infinite word. The following are equivalent:
\begin{center}
(a) $v \gess d^\infty$; \qquad (b) $v \gess d^m v$ for all $m \geq 1$; \qquad (c)  $v \gess d v$.
\end{center}
Likewise, if we replace $\gess$ with $\less$ in (a),(b),(c), the resulting three statements are equivalent to each other. \end{lem} 
\begin{proof} We will prove that (a) implies (b), and that (c) implies (a). The fact that (b) implies (c) is trivial. The proof of the corresponding statements for $\less$ is analogous.

To show that (a) implies (b), first suppose that $\|d\|$ is odd.  For all $i\ge1$, (a) implies that $v\gess d^{\infty} \gess d^{2i-1}v$, proving (b) when $m$ is odd.  Now suppose that we had $v\lesseq d^{2i} v$ for some $i$.  Then $d^{2i-1}v \less v \lesseq d^{2i} v$, which forces $v = d^\infty$, causing a contradiction. This proves (b) for even $m$. Now consider the case when $\|d\|$ is even. Suppose for contradiction that $v\lesseq d^m v $ for some $m \geq 1$.  Then $v\lesseq d^m v\lesseq d^{2m} v \lesseq \dots \lesseq d^{jm}v$ for all $j \geq 1$, contradicting (a). 

To prove that (c) implies (a), first consider the case when $\|d\|$ is even.   Since $v \gess d v$ we have that $$v \gess d v \gess d^2 v \gess \ldots \gess d^j v$$ for all $j \geq 1$, and so $v \gess d^\infty$.  
Next, we consider the case when $\|d\|$ is odd.  Suppose we had $v \lesseq d^2 v$.  Then
$dv \less v \lesseq d^2 v$,
which forces $v = d^\infty$, causing a contradiction.  Therefore, $v \gess d^2 v$, and we obtain $$v \gess d^2 v \gess d^4 v \gess \ldots \gess d^{2j} v$$ for all $j\ge 1$.   We conclude that (a) holds in all cases. 
\end{proof}

\begin{lem} \label{lem:indineq} Let $\pi \in \Al_n(\Sigma_{\sigma})$ and let $\zeta$ be a prefix defined by a valid $\sigma$-segmentation of~$\hat{\pi}^{\star}$. 
\begin{enumerate}[(a)]
\item If $\pi_n \neq n$ (so that $x$ and $p$ are defined), then $\pi_i \geq \pi_x$ implies that $(z_{[i, n-1]})^\infty \gesseq p^\infty$.  \item If $\pi_n \neq 1$ (so that $y$ and $q$ are defined), then $\pi_i \leq \pi_y$ implies that $(z_{[i, n-1]})^\infty \lesseq q^\infty$.  \end{enumerate}
\end{lem}

\begin{proof} We will prove (a); the proof of (b) is analogous.   Suppose that $\pi_n \neq n$ and $\pi_i \geq \pi_x$.  By Lemma~\ref{lem:patternsij},  $s^{(m)}$ induces $\pi$, and thus $s^{(m)}_{[i, \infty)} \gesseq s^{(m)}_{[x, \infty)}$ for $m \geq \frac{n}{2}$.   We claim that $z_{[i, n-1]} p^\infty \gesseq p^\infty$.  This is because if we had $z_{[i, n-1]} p^\infty \less p^\infty$, 
there would exist $m \geq \frac{n}{2}$ such that $z_{[i, n-1]} p^{2m} \less (p^\infty)_{[1, j-1]}$, where $j = n - i +2 m(n-x)$.  However, this is a contradiction to 
$s^{(m)}_{[i, \infty)} =z_{[i, n-1]} p^{2m} s_{[n-i + 2m(n-x), \infty)} \gesseq p^{2m + 1} s_{[(2m+1)(n-x), \infty)} = s^{(m)}_{[x, \infty)}$. Applying Lemma~\ref{dinfinity} to $z_{[i, n-1]} p^\infty \gesseq p^\infty$ gives $(z_{[i, n-1]})^\infty \gesseq p^\infty$.\end{proof}

\begin{lem} \label{lem:nocbtwn} 
Let $\zeta$ be the prefix defined by a valid $\sigma$-segmentation of $\hat{\pi}^{\star}$.  Consider three cases depending on the value of $\pi_n$:
\begin{enumerate}[(a)]
 \item If $\pi_n \notin \{1, n\}$ and $q^\infty \less w \less p^\infty$, then there is no $1 \leq c < n$ such that $qw \less z_{[c, n-1]} w \less pw$. 
\item If $\pi_n = 1$ and $ w \less p^\infty$, then there is no $1 \leq c < n$ such that $w \less z_{[c, n-1]} w\less p w$.
\item If $\pi_n = n$ and $q^\infty \less w $, then there is no $1 \leq c < n$ such that $q w \less z_{[c, n-1]} w \less w$.  \end{enumerate} \end{lem}  

\begin{proof} 
First we claim that if $q^\infty \less w \less p^\infty$ (resp. $w \less p^\infty$ if $\pi_n = 1$; and $q^\infty \less w$ if $\pi_n = n$), then
\begin{equation}\label{geqn} \pi_i > \pi_n \text{ implies } z_{[i, n-1]} w \gess w; \text{ and } 
 \pi_i < \pi_n \text{ implies } z_{[i, n-1]} w \less w. \end{equation}
Suppose that $\pi_i > \pi_n$.  By Lemma~\ref{lem:indineq}, we have $(z_{[i, n-1]})^\infty \gesseq p^\infty \gess w$.    Applying Lemma~\ref{dinfinity}, we obtain $z_{[i, n-1]} w \gess w$.  Similarly, Lemma~\ref{lem:indineq} applied to $\pi_i < \pi_n$ gives $(z_{[i, n-1]})^\infty \lesseq q^\infty \less w$, and Lemma~\ref{dinfinity}  implies that $z_{[i, n-1]} w \less w$.

In the case that $\pi_n \neq n$ (so that $x$ and $p$ are defined), we will show that if $w \less p^\infty$, then
\begin{equation} \pi_c  > \pi_x \text{ implies } z_{[c, n-1]} w \gess p w.
\end{equation}
Consider the cases $c > x$ and $c< x$ separately.  
\begin{itemize}
\item Case $c > x$.  Since $\pi_c > \pi_x$, Lemma~\ref{flipping} implies that $z_{[c, n-1]} \gesseq z_{[x, x + n - c -1]}$.  If this inequality is strict, it follows immediately that $z_{[c, n-1]} w \gess p w$.  It remains to consider the case when $z_{[c, n-1]} = z_{[x, x + n - c]}$. If $\|z_{[c, n-1]}\|$ is even (resp. odd), Lemma~\ref{flipping} applied $n-c$ times to $\pi_c > \pi_x$ implies that $\pi_n > \pi_{x + n - c}$ (resp. $\pi_n < \pi_{x + n - c}$).  By Equation~\eqref{geqn}, we have $w \gess z_{[x+n-c, n-1]} w$ (resp. $w \less z_{[x + n -c , n-1]}w$).  Prepending the equal prefixes $z_{[c, n-1]} = z_{[x, x + n - c-1]}$, which flips the inequality an even (resp. odd) number of times, we obtain $z_{[c, n-1]} w \gess p w$.

\item Case $c < x$.  Since $\pi_c > \pi_x$, Lemma~\ref{flipping} implies that $z_{[c, c + n - x - 1]} \gesseq z_{[x, n-1]}$.  If this inequality is strict, it follows immediately that $z_{[c, n-1]} w \gess pw$.  It remains to consider the case when $z_{[c, c + n - x - 1]} = z_{[x, n-1]}$.  
If $\|z_{[x, n-1]}\|$ is even (resp. odd),  Lemma~\ref{flipping} applied $n-x$ times to $\pi_c > \pi_x$ implies that $\pi_{c + n - x} > \pi_{n}$ (resp. $\pi_{c + n -x} < \pi_n$).  By Equation~\eqref{geqn}, we have $z_{[c + n - x, n-1]} w \gess w$ (resp. $z_{[c + n - x, n-1]} w \less w$).  Prepending the equal prefixes $z_{[c, c + n - x-1]} = z_{[x, n-1]}$, which flips the inequality an even (resp. odd) number of times, we obtain $z_{[c, n-1]} w \gess pw$.  
\end{itemize}

In the case that $\pi_n \neq 1$ (so that $y$ and $q$ are defined), an analogous argument shows that if $q^\infty \less w$, then $\pi_c < \pi_y$ implies $z_{[c, n-1]} w \less qw$. 
It follows that, when $\pi_n \notin \{1, n\}$, there is no $1 \leq c < n$ such that $qw \less z_{[c, n-1]} w \less pw$; and similarly for the other cases.  
\end{proof}

The following theorem describes the structure of the words in $\Int(\pi, \Sigma_{\sigma})$.

\begin{thm}\label{thm:intervals} We have $\Pat(w, \Sigma_{\sigma}, n) = \pi$ if and only if there exists a valid $\sigma$-segmentation of $\hat{\pi}^{\star}$ with associated prefix $\zeta=w_{[1, n-1]}$ such that the following condition (depending on $\pi_n$) is satisfied:
\begin{enumerate}[(a)]
\item  if $\pi_{n} \notin \{1, n\}$, then $q^\infty \less w_{[n, \infty)} \less p^\infty$;
\item if $\pi_n = 1$, then $\wmin \lesseq w_{[n, \infty)} \less p^\infty$; 
\item if $\pi_n = n$, then $q^\infty \less w_{[n, \infty)} \lesseq \wmax$.
\end{enumerate} 
\end{thm}

\begin{proof} 
Suppose first that $\Pat(w, \Sigma_{\sigma}, n) = \pi$.   Fix the valid $\sigma$-segmentation of $\hat{\pi}^{\star}$ such that $\zeta = w_{[1, n-1]}$, which is guaranteed to exist by Lemma~\ref{beginszeta}.   Since $w$ induces $\pi$, when $\pi_n \neq n$, the inequality $\pi_x > \pi_n$ implies that $p w_{[n, \infty)} \gess w_{[n, \infty)}$.  Applying Lemma~\ref{dinfinity}, we obtain $w_{[n, \infty)} \less p^\infty$. 
Similarly, when $\pi_n \neq 1$, the inequality $\pi_y < \pi_n$ implies that $qw_{[n, \infty)} \less w_{[n, \infty)}$, and applying Lemma~\ref{dinfinity} gives $q^\infty \less w_{[n, \infty)}$. The inequalities $w_{[n, \infty)} \lesseq \wmax$ 
and $\omega \lesseq w_{[n, \infty)}$ follow immediately from the fact that $\wmax$ and $\wmin$ are the largest and the smallest words in $\Wk$ with respect to $\less$, respectively. 

For the other direction, we will show that any word $w$ satisfying the above conditions induces $\pi$. Consider the case that $\pi_n \notin \{1, n\}$; the remaining cases follow analogous arguments.     Let $\zeta$ be the prefix defined by a valid $\sigma$-segmentation of $\hat{\pi}^{\star}$, and let $w$ be a word such that $w_{[1, n-1]} = \zeta$ and $q^\infty \less w_{[n, \infty)} \less p^\infty$.  We claim that $\Pat(w, \Sigma_{\sigma}, n)$ is defined.  Indeed, if $\Pat(w, \Sigma_{\sigma}, n)$ were not defined, there would be an index $1 \leq c < n$ such that $z_{[c, n-1]} w_{[n, \infty)} = w_{[n, \infty)}$.  Then $w_{[n, \infty)} = (z_{[c, n-1]})^\infty$, and so $q^\infty \less (z_{[c, n-1]})^\infty \less p^\infty$. Applying Lemma~\ref{dinfinity}, we would obtain $q (z_{[c, n-1]})^\infty \less (z_{[c, n-1]})^\infty  \less p (z_{[c, n-1]})^\infty $, equivalently $qw_{[n, \infty)} \less  z_{[c, n-1]}w_{[n, \infty)} \less p w_{[n, \infty)}$, a contradiction to Lemma~\ref{lem:nocbtwn}(a).   Therefore, $\Pat(w, \Sigma_{\sigma}, n)$ is defined and there is no $1 \leq c < n$ such that $w_{[n, \infty)} \less w_{[c, \infty)} \less w_{[x, \infty)}$.  By Lemma~\ref{lem:patternsij}, we conclude that $\Pat(w, \Sigma_{\sigma}, n) = \pi$.  
\end{proof}

\begin{cor}\label{cor:intervals}
$\Int(\pi, \Sigma_{\sigma})$ can be expressed as a finite disjoint union of intervals as
$$\Int(\pi, \Sigma_{\sigma})=\begin{cases}
\displaystyle\bigcup_\zeta \{ \zeta w_{[n, \infty)}: q^\infty \less w_{[n, \infty)} \less p^\infty\} & \text{if }\pi_{n} \notin \{1, n\},\\
\displaystyle\bigcup_\zeta \{ \zeta w_{[n, \infty)}: \wmin \lesseq w_{[n, \infty)} \less p^\infty\} & \text{if }\pi_{n}=1,\\
\displaystyle\bigcup_\zeta \{ \zeta w_{[n, \infty)}: q^\infty \less w_{[n, \infty)} \lesseq \wmax\} & \text{if }\pi_{n}=n,
\end{cases}
$$
where $\zeta$ ranges over the prefixes defined by valid $\sigma$-segmentations of $\hat{\pi}^{\star}$.
\end{cor}

\begin{proof}
From Theorem~\ref{thm:intervals} and the fact that $\hat{\pi}^{\star}$ can only have finitely many $\sigma$-segmentations, we conclude that $\Int(\pi, \Sigma_{\sigma})$ is a finite union of intervals defined by valid $\sigma$-segmentations of $\hat{\pi}^{\star}$ whose endpoints are of the form $\zeta q^\infty$ (or $\zeta \wmin$, if $\pi_n = 1$) and $\zeta p^\infty$ (or $\zeta \wmax$, if $\pi_n = n$).

This union is disjoint because different valid $\sigma$-segmentations of $\hat{\pi}^{\star}$ define distinct intervals according to Theorem~\ref{thm:intervals}.  
\end{proof}

Corollary~\ref{cor:intervals} expresses $\Int(\pi, \Sigma_{\sigma})$ as a union of intervals in $\Wk$, one for each valid $\sigma$-segmentation of $\hat{\pi}^{\star}$. We call such intervals \textit{allowed intervals}.  Let $I_n(\Sigma_{\sigma})$ denote the total number of allowed intervals of $\Sigma_{\sigma}$ for permutations $\pi$ of length~$n$.  By definition, $I_{n}(\Sigma_{\sigma})$ equals the number of pairs $(\pi, E)$ where $\pi \in \Al_n(\Sigma_{\sigma})$ and $E$ is a valid $\sigma$-segmentation of $\hat{\pi}^{\star}$.  

 Note that, since the words $\zeta= z_{[1, n-1]}$, $p= z_{[x, n-1]}$ and $q= z_{[y, n-1]}$ that appear in Corollary~\ref{cor:intervals} are defined in terms of a valid $\sigma$-segmentation of $\hat{\pi}^{\star}$,  we have that $p \neq q^2$ and $q \neq p^2$. Additionally, 
by Lemma~\ref{prefixlemsd2}, both $p$ and $q$ are primitive or equal to $d^2$ for some primitive word $d$ such that $\|d\|$ is odd. 

  Each word $w \in \Wk$ for which $\Pat(w, \Sigma_{\sigma}, n)$ is defined must belong to some allowed interval (as described in Lemma~\ref{beginszeta}), and so allowed intervals partition the set of words $w \in \Wk$ for which $\Pat(w, \Sigma_\sigma, n)$ is defined.

\begin{example} For the signed shift $\Sigma_{\sigma}$ with signature $\sigma = +-$ and patterns of length $n=3$, there are $I_3(\Sigma_{+-})=8$ allowed intervals:
\begin{align*}
\Int(123, \Sigma_{+-}) &=  (0^\infty, 0010^\infty]_{\less} \cup [0110^\infty, 01^\infty)_{\less}; \\ 
\Int(132, \Sigma_{+-}) &=  (01^\infty, (01)^\infty)_{\less};  \\
\Int(231, \Sigma_{+-}) &= ((01)^\infty, 010^\infty]_{\less} \cup [110^\infty, 1^\infty)_{\less}; \\
\Int(213, \Sigma_{+-}) &= (1^\infty, 1110^\infty]_{\less} \cup [1010^\infty, (10)^\infty)_{\less}; \\
\Int(312, \Sigma_{+-}) &= ((10)^\infty, 10^\infty)_{\less}.
\end{align*}
 For example, for $\pi = 231$, the marked cycle $\hat{\pi}^{\star} = {\star}31$ has two valid $\sigma$-segmentations $(0, 0, 3)$ and $(0, 2, 3)$, and so $\Int(231, \Sigma_{+-})$ is a union of two allowed intervals.  The segmentation $(0, 0, 3)$ gives $\zeta = 11$, $p = 11$, and its corresponding allowed interval is $$\{\zeta w_{[n, \infty)}: \wmin \lesseq w_{[n, \infty)} \less p^\infty\}=[110^\infty, 1^\infty)_{\less},$$ using that $\wmin = 0^\infty$.  The segmentation $(0, 2, 3)$ defines $\zeta' = 01$, so $p' = 01$, and the corresponding allowed interval is $$\{ \zeta' w_{[n, \infty)} : \wmin \lesseq w_{[n, \infty)} \less (p')^\infty\}=[(01)^\infty, 010^\infty)_{\less}.$$
 
By Definition~\ref{def:less}, there is no word $v$ such that $$0\wmax = 010^\infty \less v  \less110^\infty = 1 \wmax.$$ It follows that the above 8 allowed intervals partition the set of words for which $\Pat(w, \Sigma_{+-}, 3)$ is defined.  See Figure~\ref{fig:alintervals} for these intervals interpreted as intervals of the real line.  The identification here is the order-preserving bijection discussed in~\cite{KAchar}.  

\begin{figure}[htb]
\centering
\begin{tikzpicture}[scale=1.4]
\draw[very thick] (0, 0) -- (10, 0);
\draw[very thick] (0, -.25) -- (0, .25);
\draw[very thick]  (10, -.25) -- (10, .25);

\draw (0, .25) -- (0, -.25) node[below] {$0^\infty$};
\draw (3.333, .25) -- (3.333, -.25) node[below] {$01^\infty$};
\draw (4, .25) -- (4, -.25) node[below,xshift=2pt] {$(01)^\infty$};
\draw (6.666, .25) -- (6.666, -.25) node[below] {$1^\infty$};
\draw (8, .25) -- (8, -.25) node[below,xshift=6pt] {$(10)^\infty$};
\draw (10, .25) -- (10, -.25) node[below] {$10^\infty$};

\draw [thick, decorate, decoration={brace,amplitude=6pt},yshift=0pt]
(0.05,0.45) -- (3.2833,0.45) node [black,midway,yshift=0.5cm] 
{$123$};
\draw [thick, decorate, decoration={brace, amplitude = 6pt},yshift=0pt]
(3.383,0.45) -- (3.95,0.45) node [black,midway,yshift=0.5cm] {$132$};
\draw [thick, decorate, decoration={brace, amplitude = 6pt},yshift=0pt]
(4.05,0.45) -- (6.6166,0.45) node [black,midway,yshift=0.5cm] {$231$};
\draw [thick, decorate, decoration={brace, amplitude = 6pt},yshift=0pt]
(6.7166,0.45) -- (7.95,0.45) node [black,midway,yshift=0.5cm] {$213$};
\draw [thick, decorate, decoration={brace, amplitude = 6pt},yshift=0pt]
(8.05,0.45) -- (9.95,0.45) node [black,midway,yshift=0.5cm] {$312$};

\draw (2.5, 0.1)node[above]{\footnotesize $0010^\infty$}--(2.5, -.1) node[below] {\footnotesize $0110^\infty$};
\draw (5, 0.1)node[above]{\footnotesize $010^\infty$}--(5, -.1)node[below]{\footnotesize $110^\infty$};
\draw(7.5, 0.1)node[above]{\footnotesize $1010^\infty$}--(7.5, -.1) node[below]  {\footnotesize $1110^\infty$};
  
\end{tikzpicture}
\caption{The sets $\Int(\pi, \Sigma_{+-})$ for $\pi \in \Al_{3}(\Sigma_{+-})$,
interpreted as unions of allowed intervals in the real line.} 
\label{fig:alintervals}
\end{figure}
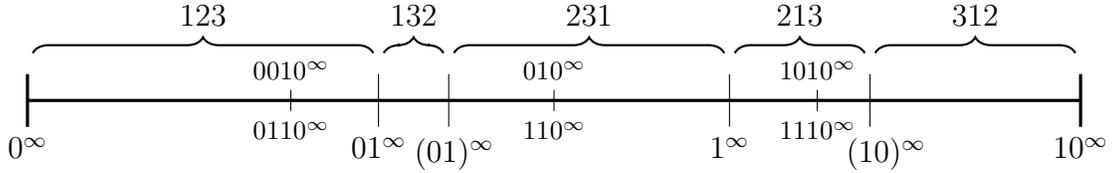

\end{example}

Recall that the number of primitive words of length $t$ on $k$ letters is 
\begin{equation}\label{eq:psi} \psi_k(t):= \sum_{d | t} \mu(d) k^\frac{t}{d},\end{equation} 
where $\mu$ denotes the number-theoretical M\"obius function. The number of words in $\Wk$  of the form  $z_{[1, n-i -1]} (z_{[n-i, n-1]})^\infty$, for some $i$ and such that $z_{[n-i, n-1]}$ is primitive, is given by:
\begin{equation}\label{ank} a(n, k) := \sum_{i = 1}^{n-1} k^{ n- i- 1} \psi_{k}(i).\end{equation}

Now we can give a complete proof of the following result, which was originally stated in \cite[Thm.~5.3]{KAchar}.

\begin{thm}\label{thm:numberalin} We have $ \card{\Al_{n}(\Sigma_{\sigma})} \leq  I_{n}(\Sigma_{\sigma}) $. Additionally, 
\begin{equation}\label{eq:ank}I_{n}(\Sigma_{\sigma})= \begin{cases}
a(n, k) + (k - 2)k^{n-2} & \mbox{if } \sigma_0 = \sigma_{k-1} = +,\\
a(n, k) + (k-1)k^{n-2} & \mbox{if } \sigma_0 \neq \sigma_{k-1},\\
a(n, k) + (k^2 - 2)k^{n-3} & \mbox{if } \sigma_0 = \sigma_{k-1} = -,
\end{cases} \end{equation}
with $a(n,k)$ as given by Equation~\eqref{ank}.
\end{thm}

\begin{proof} To show the first statement, recall that the number of allowed intervals $I_{n}(\Sigma_{\sigma})$ equals the number of pairs $(\pi, E)$ where $\pi \in \Al_n(\Sigma_{\sigma})$ and $E$ is a valid $\sigma$-segmentation of $\hat{\pi}^{\star}$.    Since each $\pi \in \Al_n(\Sigma_{\sigma})$ has at least one valid $\sigma$-segmentation by Theorem~\ref{characterization}, it is clear that $\card{ \Al_n(\Sigma_\sigma)} \leq I_n(\Sigma_{\sigma})$.

By Corollary~\ref{cor:intervals}, allowed intervals are of the form $\{\zeta w_{[n, \infty)}: q^\infty \less w_{[n, \infty)} \less p^\infty\}$, $\{\zeta w_{[n, \infty)}: \wmin \lesseq w_{[n, \infty)} \less p^\infty\}$, or $\{\zeta w_{[n, \infty)}: q^\infty \less w_{[n, \infty)} \lesseq \wmax\}$, where $\zeta$, $p$ and $q$ are defined by a valid $\sigma$-segmentation of $\hat{\pi}^{\star}$ for some permutation $\pi \in \Al_n(\Sigma_{\sigma})$.  We enumerate the allowed intervals by counting their endpoints of the form $\zeta p^\infty$ and $\zeta \wmax$, of which each interval has precisely one.   

We claim that endpoints of the form $\zeta p^\infty$ are precisely the elements of the set 
$$\Delta^{\sigma}_n = \{ uv^\infty : |u|+|v|=n-1, v \text{ is primitive, and }v^\infty \neq \wmin\}.$$  Indeed, for an endpoint $\zeta p^\infty$, Lemma~\ref{prefixlemsd2} and the proof of Lemma~\ref{patdefined} imply that $p^\infty \neq \wmin$, but all other words $p$ are possible as long as either $p$ is primitive, or $p = d^2$ with $d$ primitive and $\|d\|$ is odd. To express $\zeta p^\infty$ as an element of $\Delta^{\sigma}_n$, we consider these two cases separately:  
\begin{itemize}
\item If $p$ is primitive, we take $u=z_{[1, x-1]}$ and $v=p$. In this case, if $\| v \|$ is odd, then $v$ cannot be a suffix of $u$, since otherwise Lemma~\ref{lem:doubleback} would imply that that $q = p^2$, contradicting that $\zeta$ was given by a valid segmentation.
\item If $p = d^2$, we take $u=z_{[1, x-1]}d$ and $v=d$.  Note that $\|v\|$ is odd and $v$ is a suffix of $u$.  
\end{itemize}
Conversely, each word in $\Delta^{\sigma}_n$ corresponds to a unique endpoint $\zeta p^\infty$ by taking $\zeta = uv$, and $p = vv$ if $v$ is a suffix of $u$ and $\|v\|$ is odd, or $p = v$ otherwise.  To count the number of words in $\Delta^{\sigma}_n$, we consider three cases:

\begin{itemize}
\item If $\sigma_0 = +$, then $\wmin = 0^\infty$ by Equation~\eqref{eq:wmax}.  Therefore, we must have $v \neq 0$, and so $\card{\Delta^{\sigma}_n} = a(n, k) - k^{n-2}$. 

\item If $\sigma_0 = -$ and $\sigma_{k-1}= +$, then $\wmin = 0(k{-}1)^\infty$.  Since $\wmin \neq v^\infty$ for all $v$, there are no restrictions on the choice of $v$.  Hence, $\card{\Delta^{\sigma}_{n}} = a(n, k)$.  

\item If $\sigma_0 = -$ and $\sigma_{k-1} = -$, then $\wmin =  (0(k{-}1))^\infty$.  Therefore, we must have $v \neq 0(k{-}1)$, and so $|\Delta^{\sigma}_n| = a(n, k) - k^{n-3}$.
\end{itemize}

\smallskip
Next we count endpoints of the form $\zeta \wmax$, which correspond to allowed intervals for $\pi \in \Al_n(\Sigma_\sigma)$ satisfying $\pi_n = n$.  

\begin{itemize}
\item If $\sigma_{k-1} = +$, then $\wmax = (k{-}1)^\infty$ by Equation~\eqref{eq:wmax}.  We claim that $z_{n-1} \neq k{-}1$.
To see this, let $\zeta w_{[n, \infty)}$ be a word in the corresponding allowed interval, and so $\Pat(\zeta w_{[n, \infty)}, \Sigma_{\sigma}, n) = \pi$ with $\pi_n = n$.  Then the inequality $\pi_{n-1} < \pi_n$ implies that $z_{n-1}w_{[n, \infty)} \less w_{[n, \infty)}$, and by Lemma~\ref{dinfinity}, $(z_{n-1})^\infty \less w_{[n, \infty)}$.  In particular,  we have $(z_{n-1})^\infty \less \wmax = (k{-}1)^\infty$.     Since all other prefixes $\zeta$ are possible, the number of endpoints of the form $\zeta \wmax$ is equal to $(k-1)k^{n-2}$.

\item If $\sigma_0 = +$ and $ \sigma_{k-1} = -$, then $\wmax = (k{-}1)0^\infty$.  Since $\wmax \neq (z_{[i, n-1]})^\infty$ for any $z_{[i, n-1]}$, all prefixes $\zeta$ are possible.   Therefore, the total number of endpoints of the form $\zeta \wmax$ is equal to $k^{n-1}$.  

\item If $\sigma_0 = \sigma_{k-1} = -$, then $\wmax = ((k{-}1)0)^\infty$ by Equation~\eqref{eq:wmax}.  We claim that $z_{[n-2, n-1]} \neq (k{-}1)0$.  To see this, let $\zeta w_{[n, \infty)}$ be a word in the corresponding allowed interval, and so $\Pat(\zeta w_{[n, \infty)}, \Sigma_{\sigma}, n) = \pi$ with $\pi_n = n$.  Then $\pi_{n-2} < \pi_n$ implies that $z_{[n-2, n-1]} w_{[n, \infty)} \less w_{[n, \infty)}$, and by Lemma~\ref{dinfinity}, $(z_{[n-2, n-1]})^\infty \less w_{[n, \infty)}$.  In particular, $(z_{n-2} z_{n-1})^\infty \less \wmax = ((k{-}1)0)^\infty$, and so $z_{[n-2, n-1]}\neq (k{-}1)0$.  Since all other prefixes are possible, the number of endpoints of the form $\zeta \wmax$ is equal to $(k^2 -1)k^{n-3}$.   
\end{itemize}

 Combining the words in each case gives Equation~\eqref{eq:ank}.  
\end{proof}

\section{The Negative Shift}\label{sec:negativeshift}

Restricting to the positive and negative shifts, Theorem~\ref{characterization} allows us to derive simple formulas for the smallest positive integer $k$ such that $\pi$ is realized by the $k$-shift, and similarly for the $-k$-shift. In the rest of the paper, we will use $\Sigma_k$ and $k$-segmentation (resp. $\Sigma_{-k}$ and $-k$-segmentation) to refer to $\Sigma_\sigma$ and a $\sigma$-segmentation where $\sigma = +^k$ (resp. $\sigma = -^k$). 

For the positive shift, let  
$N(\pi) = \text{min}\{k\ge2 : \pi \in \Al(\Sigma_k)\}$.
It was shown in~\cite{EliShift} that  
\begin{equation}\label{eq:N}N(\pi) = 1 + \text{des}(\hat{\pi}^{\star}) + \epsilon(\hat{\pi}^{\star}),
\end{equation}
where $\epsilon(\hat{\pi}^{\star}) = 1$ if $\pi_{n-1}\pi_n = 21$ or $\pi_{n-1} \pi_n = (n-1)n$, and $\epsilon(\hat{\pi}^{\star}) = 0$ otherwise.  This formula can be deduced from Theorem~\ref{characterization} by noticing that each descent of $\hat{\pi}^{\star}$ requires a new index in the segmentation, and an additional index is needed when conditions (b) or (c) in Definition~\ref{def:segmentation} hold, together with the fact that all positive segmentations are valid by Lemma~\ref{prefixlemsd2}. Note that each $\hat{\pi}^{\star}$ has a unique $N(\pi)$-segmentation.  

The analogous definition for the negative shift is
\begin{equation}\label{eq:barN} \N(\pi) = \min\{k\ge2 : \pi \in \Al( \Sigma_{-k} ) \}.
\end{equation}
Using Theorem~\ref{characterization}, we try to construct a valid $-k$-segmentation for $\hat\pi$ with the smallest possible $k$. 
An index in the segmentation is needed for each ascent of $\hat{\pi}^{\star}$. If conditions (d) or (e)  in Definition~\ref{def:segmentation} do not apply, a $-k$-segmentation exists as long as $k \geq 1 + \asc(\hat{\pi}^{\star})$. In this case, for $k=1 + \asc(\hat{\pi}^{\star})$, there is a unique $-k$-segmentation, whose indices are the ascents of $\hat{\pi}^{\star}$. We call this the \textit{minimal negative segmentation} of $\hat\pi$.  However, it can happen that no such minimal negative segmentation exists (because conditions (d) or (e) apply), or that the minimal negative segmentation is invalid.

\begin{definition}\label{def:ccr} We say that $\pi$ is 
\begin{itemize}
\item {\em $\vee$-cornered} if $\pi_{n-2}\pi_{n-1}\pi_n = (n-1)1n$, \\ (equivalently, we invoke condition~(d) in Definition~\ref{def:segmentation});
\item {\em $\wedge$-cornered} if $\pi_{n-2}\pi_{n-1}\pi_n = 2n1$, \\ (equivalently, we invoke condition~(e) in Definition~\ref{def:segmentation});
\item {\em collapsed} if the minimal negative segmentation of $\hat{\pi}^{\star}$ exists but is invalid;
\item {\em regular} if $\pi$ is neither cornered nor collapsed.  
\end{itemize} 
\end{definition}

We say that $\pi$ is {\em cornered} if it is either $\vee$-cornered or $\wedge$-cornered.
Note that a permutation cannot be simultaneously cornered and collapsed. Indeed, a collapsed permutation requires the words $p$ and $q$ to be defined, which only happens if $\pi_n \notin\{1, n\}$. 
 
 \begin{lem}\label{lem:formcollapsed} 
 Suppose that $n-x = 2(n-y)$ and $n-y$ is odd.  Then $\pi$ is collapsed if and only if  $\pi_{x+j} - \pi_{y+j} = (-1)^j$ for all $1 \leq j \leq n-y$.   The above statement also holds when $x$ and $y$ are exchanged.  
 \end{lem}
 \begin{proof}  Suppose that $\pi$ is collapsed and let $c = n-y$.  Since $\pi_n \notin \{1, n\}$, conditions (d) and (e) in Definition~\ref{def:segmentation} do not apply, and so a minimal negative segmentation of $\hat{\pi}^{\star}$ exists. Let $\zeta$ be the prefix defined by the minimal negative segmentation $E$, of $\hat{\pi}^{\star}$ and note that $p = q^2$.  For $j = c$, the conclusion $\pi_{x+c} - \pi_{y+c} = \pi_{y} - \pi_n =(-1)^{c}$ is immediate.  It remains to show that $\pi_{x+j} - \pi_{y + j} = (-1)^j$ for all $1 \leq j < c$.  
 
We claim that, for all $1 \leq j < c$, the values $\pi_{n-j}$ and $\pi_{n-j-c}$ are consecutive.  Let us assume that $\pi_{n-j} < \pi_{n-j-c}$ (the other case follows a parallel argument), and suppose for contradiction that there is an index $i$ such that $\pi_{n-j} < \pi_i < \pi_{n-j-c}$.  Since $z_{[n-j, n-1]} = z_{[n-j-c, n-c-1]}$, Lemma~\ref{flipping} applied $i$ times yields $\pi_{n} < \pi_{i+j} < \pi_{n-c} = \pi_x$ or $\pi_y = \pi_{n-c} < \pi_{i + j} < \pi_{n}$ (depending on the parity of $j$), a contradiction to $\pi_x = \pi_n + 1$ or $\pi_{y} = \pi_n -1$, respectively, thus proving the claim. 
 
Note that $\hat{\pi}^{\star}_{\pi_y} \hat{\pi}^{\star}_{\pi_n} \hat{\pi}^{\star}_{\pi_x} = \pi_{y+1} {\star}\pi_{x+1}$. Since $z_{[x, y-1]} = z_{[y, n-1]}=q$, we have that $z_{x} = z_y$ and $\pi_y < \pi_x$. By Lemma~\ref{flipping}, we obtain $\pi_{y+1} > \pi_{x+1}$, and so $\pi_{y+1} {\star}\pi_{x+1}$ is a descent.  By the previous paragraph, $\pi_{y+1}$ and $\pi_{x+1}$ are consecutive and so $\pi_{x + 1} - \pi_{y+1} = -1$.  Suppose inductively that $\pi_{x + j} - \pi_{y+j} = (-1)^j$. Since $z_{[y + j + 1]} = z_{[x + j + 1]}$ and $(-1)^j(\pi_{x + j} -\pi_{y+j}) > 0$, Lemma~\ref{flipping} implies that $(-1)^{j+1}(\pi_{x + j+1} - \pi_{y+j+1}) > 0$.  By the previous paragraph, $\pi_{y + j+1}$ and $\pi_{x + j + 1}$ are consecutive and so $\pi_{x + j + 1} - \pi_{y + j +1} = (-1)^{j+1}$.

To prove the converse, suppose now that $\pi_{x + j} - \pi_{y + j} = (-1)^j$ for all $1 \leq j <c$. 
Since $\pi_{x + 1} = \pi_{y+1} - 1$, the entries $\hat{\pi}^{\star}_{\pi_y} \hat{\pi}^{\star}_{\pi_{n}} \hat{\pi}^{\star}_{\pi_x}= \pi_{y+1} {\star}\pi_{x+1}$ do not contain an ascent.  Therefore, in the minimal negative segmentation $E$ of $\hat{\pi}^{\star}$, there is no index between $\pi_y$ and $\pi_x$ (i.e., there is no $e_t$ such that $\pi_y \leq e_t < \pi_x$), and so by Definition~\ref{def:segmentation}, we have $z_{x} = z_y$.  For $1 \leq j < c$, we will show $z_{x + j} = z_{y+j}$ following a similar argument.  

If $j$ is odd, then $\pi_{y + j} = \pi_{x+j} + 1$ and $\pi_{y + j + 1} = \pi_{x + j + 1} - 1$.  Since $\hat{\pi}^{\star}_{\pi_{y+j}} \hat{\pi}^{\star}_{\pi_{x+j}} = \pi_{y + j + 1} \pi_{x + j + 1}$ is a descent, there is not an index in $E_{\pi}$ between $\pi_{x + j}$ and $\pi_{y + j}$.  By Definition~\ref{def:segmentation}, we have $z_{x + j} = z_{y + j}$.  If $j$ is even, the argument is the same except that the indices $x$ and $y$ are switched.  
Therefore, $z_{[x, y-1]} = z_{[y, n-1]}$ and so $p = q^2$, which means that $\pi$ is collapsed.
\end{proof}
  
Now we can deduce from  Theorem~\ref{characterization} the following analogue of Equation~\eqref{eq:N} for negative shifts.\footnote{After we posted an earlier preprint including this result on \texttt{arxiv.org}, we were informed by Charlier and Steiner that they independently obtained Theorem~\ref{numberseg} and Corollary~\ref{cor:smallest}. Their work was subsequently posted in~\cite{ECWSnegbeta}.}

\begin{thm} \label{negativeshift}\label{numberseg}  
With $\N(\pi)$ defined by Equation~\eqref{eq:barN}, we have
$$\N(\pi) = 1 + \asc(\hat{\pi}^{\star}) + \overline{\epsilon}(\hat{\pi}^{\star}),$$
where $\overline{\epsilon}(\hat{\pi}^{\star}) = 1$ if $\pi$ is cornered or collapsed; and $\overline{\epsilon}(\hat{\pi}^{\star}) = 0$ if $\pi$ is regular.  Additionally,  the number of valid $-\N(\pi)$-segmentations of $\hat{\pi}^{\star}$ is $1$ if $\pi$ is regular, $2$ if $\pi$ is cornered, and $\min\{n - x, n-y\}$ if $\pi$ is collapsed.
\end{thm}

\begin{proof} If $\pi$ is regular, then the minimal negative segmentation is the $-\N(\pi)$-segmentation of $\hat{\pi}^{\star}$.   If $\pi$ is cornered, then either part (d) or part (e) of Definition~\ref{def:segmentation} applies, requiring an additional index which is not an ascent of $\hat\pi$. Choosing either $e_1 = 0$ or $e_{k-1} = n-1$ defines two distinct segmentations, which are valid because $\pi_n \in \{1, n\}$ in each case.

If $\pi$ is collapsed, then the minimal negative segmentation $E'$ is not valid, and we require an additional index which is not an ascent of $\hat{\pi}^{\star}$. Letting $c = \min \{ n-x, n-y\}$, the unique prefix $\zeta$ defined by $E'$ satisfies $z_{[n-2c, n-c-1]} = z_{[n - c, n-1]}$, thus giving $c$ pairs of equal letters, $z_{n-j} = z_{n-c-j}$ for $1\le j\le c$.  By placing an index $e_t$ such that $\pi_{n-j} < e_{t} \leq \pi_{n-c-j}$ (or $\pi_{n-j} < e_{t} \leq \pi_{n-c-j}$) for some $j$, the corresponding prefix $\zeta'$ satisfies $z'_{n-j} \neq z'_{n-j-c}$, giving a valid $-(2+ \asc(\hat{\pi}^{\star}))$-segmentation of $\hat{\pi}^{\star}$.  
By Lemma~\ref{lem:formcollapsed}, the values $\pi_{n-j}$ and $\pi_{n-c-j}$ are consecutive for all $1 \leq j < c$.  It follows that, for each $j$ with $1 \leq j < c$, there is exactly one choice of $e_t$ between $\pi_{n-j}$ and $\pi_{n-c -j}$, which forces $z'_{n-j} \neq z_{n-j-c}'$ in the associated prefix.  Similarly, when $j = c$, there is only one choice of index between $\pi_y$ and $\pi_{x}$, since condition~(f) in Definition~\ref{def:segmentation} requires that $e_t \neq \pi_n$.  This gives a total of $c$ choices for $e_t$, hence there are $c$ distinct prefixes arising from valid $-\N(\pi)$-segmentations of $\hat{\pi}^{\star}$.  
\end{proof}

\begin{example} Let $\pi = 3651742$. Then $\hat{\pi}^{\star} = 7{\star}62154$ has minimal negative segmentation $E' = (e'_0, e'_1, e'_2) = (0, 5, 7)$.  We write $\hat{\pi}^{\star}_{E'} = |7{\star}621|54|$.  The segmentation $E'$ defines the prefix $\zeta = 010010$, which yields $p = (010)^2 = q^2$. By Theorem~\ref{characterization}, $\pi$ is not realized by the $-2$-shift. We claim $\N(\pi)= 3$ in accordance with Theorem~\ref{numberseg}.  Indeed, we may obtain a valid $-3$-segmentation by placing an additional index to separate one of the three pairs of equal letters $z_{i} = z_{i + 3}$ for $i = 1, 2, 3$.  The distinct prefixes defined by the possible $-3$-segmentations are $\zeta^{(1)} = 121021$, $\zeta^{(2)} = 021020$ and $\zeta^{(3)} = 010020$.
\end{example}

\begin{cor}\label{cor:smallest} The smallest forbidden patterns of the $-k$-shift have length $k +2$, and there are always exactly $4$ of them. 
\end{cor} 

\begin{proof} We will show that there are exactly $4$ permutations $\pi \in \S_{n}$ such that $\N(\pi) > n-2$, and in each case we have $\N(\pi) = n-1$. These permutations are  $12\ldots (n{-}1)n$, $n(n{-}1)\ldots 21$, $12\ldots (n{-}2)n(n{-}1)$ and $n(n{-}1) \ldots 312$. It will follow $\N(\pi) \leq k$ for every $\pi \in \S_{k+1}$, but $\N(\pi)=k+1$ for exactly four patterns $\pi\in\S_{k+2}$, and so these are the smallest forbidden patterns of the $-k$-shift.

By Theorem~\ref{negativeshift}, the only way for $\pi \in \S_{n}$  to have $\N(\pi) > n-2$ is if $\asc(\hat{\pi}^{\star})=n-2$, or if $\asc(\hat{\pi}^{\star})\ge n-3$ and $\overline{\epsilon}(\hat{\pi}^{\star}) = 1$.
Suppose first that $\hat{\pi}^{\star} \in \mathcal{C}_n^{\star}$ has $n-2$ ascents, and thus no descents.  Consider two cases depending on the value of $\pi_1$, which is the entry in $\hat{\pi}^{\star}$ that is replaced by~$\star$.  
\begin{itemize}
\item Case $\pi_1 = 1$. Let $j = \pi_n$, and so $\hat{\pi}^{\star}_j = \star$.   Since $\hat{\pi}^{\star}$ does not have any descents,  we must have $\hat{\pi}^{\star}= 234\ldots j {\star}(j{+}1)\ldots n$.
Since $\hat{\pi}^{\star}$ is a cycle, we must have $j = n$.  Thus, $\hat{\pi}^{\star} = 23\ldots n \star$, and  so $\pi = 12\ldots (n{-}2)(n{-}1)n$ and $\N(\pi) = n{-}1$.  
\item Case $\pi_1 \neq 1$. Since $\hat{\pi}^{\star}_{1} \neq 1$ because $\hat{\pi}^{\star}$ is a cycle, the only possibility  is that $\hat{\pi}^{\star}_1 \hat{\pi}^{\star}_2 = {\star}1$, since otherwise the index $j$ for which $\hat{\pi}^{\star}_j = 1$ would be a descent of $\hat\pi$. Moreover, using that $\hat{\pi}^{\star}$ has no descents and no fixed points, the only possibility is that $\hat{\pi}^{\star} =  {\star}12\ldots (n{-}1)$, and so $\pi = n(n{-}1)(n{-}2) \ldots 21$ and $\N(\pi) = n-1$.  
\end{itemize}

Next we consider the cases when $\overline{\epsilon}(\hat{\pi}^{\star}) = 1$, which means that $\pi$ is cornered or collapsed.  If $\pi \in \S_{n}$ is $\vee$-cornered, we have $\hat{\pi}^{\star}_1 \hat{\pi}^{\star}_2 = \star n$ and $\hat{\pi}^{\star}_{n} = 1$.  Since $1, 2$ and $n-1$ are not ascents of $\hat{\pi}^{\star}$, it follows that  $\asc(\hat{\pi}^{\star})\le n-4$.  A similar argument applies to $\wedge$-cornered permutations.  We conclude that there are no cornered permutations $\pi \in \S_{n}$ such that $\N(\pi) = 2 + \asc(\hat{\pi}^{\star}) > n-2$.  

Finally, suppose that $\pi \in \S_{n}$ is a collapsed permutation with $\asc(\hat{\pi}^{\star})\ge n-3$, and thus $\des(\hat{\pi}^{\star})\le 1$.  
Suppose first that the minimal negative segmentation $E = (e_0, e_1, \ldots, e_{n-1})$ defines a prefix with $q = p^2$.  For each pair of equal letters  $z_{n-x-i} = z_{n-i} = j$ (where $1 \leq i \leq n-x$) in the prefix $\zeta$ defined by $E$, we have that $e_j < \pi_{n-x-i}, \pi_{n-i} \leq e_{j+1}$, and so $e_{j+1} - e_j > 1$.  In particular, $e_j + 1$ is a descent of $\hat{\pi}^{\star}$, since the indices in the minimal negative segmentation occur at ascents.  Since $\hat{\pi}^{\star}$ has at most one descent, there is at most one pair of equal letters $z_{n-x-i} = z_{n-i} = j$.  Thus, the fact that $\pi$ is collapsed and $q = p^2$ implies that $|p| = 1$.  Therefore, $\pi_{n-2}\pi_{n-1}\pi_{n} = (\pi_{n}-1) (\pi_{n} + 1) \pi_{n}$, and so $\hat{\pi}^{\star}_{\pi_{n}-1} \hat{\pi}^{\star}_{\pi_{n}} \hat{\pi}^{\star}_{\pi_{n}+1} =  (\pi_{n}+1){\star}\pi_{n}$.  Since $\hat{\pi}^{\star}$ must have $n-3$ ascents and no fixed points because $\hat{\pi}^{\star}$ is a cycle, we must have $\hat{\pi}^{\star} = 2 3 \ldots n\star (n{-}1)$.  Hence, $\pi = 12 \ldots (n{-}2)n(n{-}1)$ and $\N(\pi) = n-1$.  A parallel argument implies that $\pi = n(n{-}1)(n{-}2) \ldots 312$ is the only collapsed permutation with $p = q^2$ and $\asc(\hat{\pi}^{\star})\ge n-3$, and $\N(\pi) = n-1$ in this case as well. \end{proof}

For comparison with Corollary~\ref{cor:smallest}, the analogous result for the $k$-shift, proved in \cite[Thm.~5.13]{EliShift}, states that its smallest forbidden patterns have length $k+2$ and there are exactly $6$ of them.  

\begin{example} The smallest forbidden patterns of the $-4$-shift are $123456$, $654321$, $123465$, $654312$.  The smallest forbidden patterns of the $4$-shift are $615243$, $324156$, $342516$, $162534$, $453621$, $435261$.  
\end{example}

\section{Enumeration for the Negative Shift}\label{sec:enumeration}
The exact enumeration of patterns of length $n$ realized by the $-k$-shift is more complicated than in the positive case (which was solved in~\cite{EliShift}), since the same permutation $\pi$ may correspond to multiple allowed intervals for the $-\N(\pi)$-shift, coming from different segmentations.  When $\pi$ is cornered or collapsed, there is more than one valid $-\N(\pi)$-segmentation of $\hat{\pi}^{\star}$, as shown in Theorem~\ref{negativeshift}. For the case $k=2$, a formula for $\Al_n(\Sigma_{-2})$ was given by Archer \cite[Thm.~5.11]{KAchar}.  

\begin{definition}\label{def:canonical}
Given a permutation $\pi \in \Al_n(\Sigma_{-k})$, the {\em canonical segmentation} of $\hat{\pi}^{\star}$, denoted by $E_{\pi}$, is the $-\N(\pi)$-segmentation described in each case below.  
\begin{itemize}
\item  If $\pi$ is regular, $E_{\pi}$ is the unique $-\N(\pi)$-segmentation of $\hat{\pi}^{\star}$.
\item If $\pi$ is cornered, $E_{\pi}$ is the $-\N(\pi)$-segmentation of $\hat{\pi}^{\star}$ with $e_1 = 0$.
\item If $\pi$ is collapsed,  $E_{\pi}$ consists of the indices of the minimal negative segmentation of $\hat{\pi}^{\star}$ plus the index $\pi_n - 1$.  
\end{itemize}
\end{definition}

\begin{example}
\begin{enumerate}[(a)]
\item For the regular permutation $\pi = 436125$, we obtain $\hat{\pi}^{\star} = 2563{\star}1$, and so $\N(\pi) = 3$ by Theorem~\ref{negativeshift}.  By Definition~\ref{def:canonical}, the canonical segmentation of $\pi$ is the unique $-3$-segmentation of $\hat{\pi}^{\star}$, and so $E_{\pi} = (0, 1, 2, 6)$ and $\hat{\pi}^{\star}_{E_{\pi}} = |2|5|63{\star}1|$.  
\item For the $\vee$-cornered permutation $\pi = 32415$, we obtain $\hat{\pi}^{\star} = 5421{\star}$, and so $\N(\pi) = 2$ by Theorem~\ref{negativeshift}.  There are two $-2$-segmentations of $\hat{\pi}^{\star}$, and  by Definition~\ref{def:canonical}, $E_{\pi}$ is the one with $e_1 =0$.  Therefore, $E_{\pi} = (0, 0, 5)$ and we write $\hat{\pi}^{\star}_{E_{\pi}} = \|5421{\star}|$.  
\item For the collapsed permutation $\pi= 3715624$, we obtain $\hat{\pi}^{\star} = 547{\star}621$.   The minimal negative segmentation of $\hat{\pi}^{\star}$ is $E' = (0, 2, 7)$, defining $\zeta = 110110$.  By Definition~\ref{def:canonical}, $E_{\pi}$ is $E'$ together with the index $\pi_n - 1 = 3$, that is, $E_{\pi} = (0, 2, 3, 7)$, which we represent as $\hat{\pi}^{\star}_{E_{\pi}} = |54|7|{\star}621|$.  Note that the index $\pi_n -1$ produces a bar immediately to the left of~$\star$ in $\hat{\pi}^{\star}$.  The prefix defined by $E_{\pi}$ is $\zeta = 120220$, which satisfies $2=z_{x} \neq z_y=1$.  
\end{enumerate}
\end{example}

\begin{remark}\label{rem:zxzy} If $\pi$ is collapsed, $E$ is a valid $-k$-segmentation of $\hat{\pi}^{\star}$, and $\zeta$ is the prefix defined by $E$, then $E_{\pi} \subseteq E$ if and only if $z_x \neq z_y$.   To see this, first suppose that $E_{\pi} \subseteq E$, which implies that $E$ contains the index $\pi_n - 1$.  Since $\pi_y = \pi_n - 1 < \pi_x$, Definition~\ref{def:segmentation} forces $z_y \neq z_x$.  On the other hand, if $E_{\pi} \not \subseteq E$, then $\pi_n - 1$ cannot an index of $E$, since all other indices in $E_{\pi}$ are ascents and thus must be in $E$.   By Definition~\ref{def:segmentation}(f), $\pi_n$ is not an index of $E$ either, and so $E$ contains no index between $\pi_y$ and $\pi_x$ and we conclude that $z_x = z_y$.  
\end{remark}

We say that a  $-j$-segmentation $E= (e_0, e_1,  \ldots, e_j)$ is a {\em refinement} of a $-t$-segmentation  $E' = (e'_0, e'_1 ,\ldots ,  e'_t)$ if $\{e'_0, e'_1 ,\ldots ,  e'_t\} \subseteq \{e_0, e_1,  \ldots, e_j\}$.  With some abuse of notation, we write $E'\subseteq E$ to denote this inclusion of sets.

In Lemmas~\ref{lem:corneredenum} and~\ref{lem:collapsedenum}, we determine the cardinalities of the following two sets, respectively:
\begin{multline}\label{eq:corneredenum} 
C_{n, k}  = \{(\pi,  E): \pi\in\Al_n(\Sigma_{-k}) \text{ is cornered, } \\  \text{ and } E \text{ is a valid $-k$-segmentation of $\hat{\pi}^{\star}$ such that } E_{\pi} \not \subseteq E \},
\end{multline}
\begin{multline}\label{eq:collapsedenum} 
D_{n, k}  = \{(\pi, E): \pi\in\Al_n(\Sigma_{-k})  \text{ is collapsed, }  \\  \text{ and } E \text{ is a valid $-k$-segmentation of $\hat{\pi}^{\star}$ such that } E_{\pi} \not \subseteq E\}. 
\end{multline}

\begin{lem} \label{lem:corneredenum} For $C_{n, k}$ defined in Equation~\eqref{eq:corneredenum}, we have
$$ |C_{n, k}| = 2 \sum_{j =1}^{k-1} j^{n-3}.$$
\end{lem}

\begin{proof} Let $C_{n, k}^{\vee}$ be the set of pairs $(\pi, E)\in C_{n, k}$ such that $\pi$ is $\vee$-cornered. The remaining pairs in $C_{n,k}\setminus C_{n, k}^{\vee}$ are $\wedge$-cornered. Thus, by symmetry, $|C_{n, k}| = 2|C_{n, k}^{\vee}|$.  We will enumerate $C_{n, k}^{\vee}$. 

For a $\vee$-cornered permutation $\pi$, the proof of Theorem~\ref{negativeshift} shows that the two $-\N(\pi)$-segmentations of $\hat{\pi}^{\star}$ consist of the positions of the ascents of $\hat{\pi}^{\star}$, plus either $e'_1 = 0$ or $e'_{\N(\pi)-1} = n-1$.  The canonical $-\N(\pi)$-segmentation is the one with $e'_1 = 0$, and the remaining indices (other than $e'_0=0$ and $e'_{\N(\pi)} = n$) are the ascents of $\hat{\pi}^{\star}$. Note also that every $-k$-segmentation of $\hat{\pi}^{\star}$ is valid because $\pi_n =n$.

We claim that there is a bijection between $C_{n, k}^{\vee}$ and the set
$$\Gamma_{n, k} =  \{ z_{[1, n-1]} : 0 = z_{n-1} \leq  z_{i}  \leq z_{n-2} \leq k-2 \text{ for all } 1 \leq i \leq n-3 \}.$$
Let $(\pi, E)\in C_{n, k}^{\vee}$. Since $E$ must have indices at the ascents of $\hat{\pi}^{\star}$, the condition $E_{\pi} \not \subseteq E$ forces $e_1 > 0$, and so $E$ has $e_{k-1} = n-1$ by Definition~\ref{def:segmentation}(d).  Letting $\zeta = z_{[1, n-1]}$ be the prefix defined by $E$,
it follows that $0 \leq z_i \leq k-2$ for all $i$, and also $z_{n-1} = 0$ because $e_{1} >0$ and $\pi_{n-1}=1$.  Finally, since $\pi_{n-2}= n-1$ and $\pi_n = n$, we have that $z_{i} \leq z_{n-2}$ for all $i$, by Lemma~\ref{flipping}.  Thus, $\zeta = z_{[1, n-1]} \in \Gamma_{n,k}$.

Conversely, given a word $w_{[1, n-1]} \in \Gamma_{n,k}$, let $w = w_{[1, n-1]}((k{-}1)0)^\infty$. Then $\Pat(w, \Sigma_{-k}, n) = \pi$ for some $\vee$-cornered permutation $\pi$.  By Lemma~\ref{beginszeta}, there exists a unique valid $-k$-segmentation $E = (e_0, e_1, \ldots, e_k)$ of $\hat\pi^\star$ such that $\zeta = w_{[1, n-1]}$. 
We have that $E_{\pi} \not \subseteq  E$ because $\zeta \in \Gamma_{n, k}$ implies that $z_{n-1} = 0$, and so, as in the proof of Lemma~\ref{beginszeta}, we have $e_{1}  = \card{\{i : 1 \leq i \leq n \text{ and } z_{i} < 1\}} > 0$.  
Thus, $(\pi, E)\in C_{n, k}^{\vee}$.  

We now have that $\card{C_{n, k}} = 2\card{C_{n, k}^{\vee}} = 2\card{\Gamma_{n, k}}$ and 
$$\card{\Gamma_{n, k}} = \sum_{j = 1}^{k-1} j^{n-3},$$
which is obtained by counting the number of words in $\Gamma_{n, k}$ according to the value of $z_{n-2}$.  
\end{proof}  

Recall the definition of $\psi$ from Equation~\eqref{eq:psi}.

\begin{lem} \label{lem:collapsedenum} For $D_{n, k}$ defined in Equation~\eqref{eq:collapsedenum}, we have 
$$|D_{n, k}| = 2 \sum_{\substack{c = 1 \\ \text{odd}}}^{\lfloor \frac{n-1}{2}\rfloor} \sum_{j =1}^{k-1}  {c + k - j -2 \choose k -j} j^{n-2c-1}\psi_j(c).$$
 \end{lem}

\begin{proof} Let $D_{n, k}^{<, c}$ (resp. $D_{n, k}^{>, c}$) be the set of pairs $(\pi, E) \in D_{n, k}$ such that $\pi$ satisfies $x < y$ and $n-y  = c$ (resp. $x > y$ and $n - x = c$). By symmetry, $|D_{n, k}^{<, c}| = |D_{n, k}^{>, c}|$. Since $c \leq \frac{n-1}{2}$ and $c$ must be odd by Lemma~\ref{prefixlemsd2}, we have that \begin{equation} \label{eq:dnk} |D_{n, k}| = 2\sum_{\substack{c = 1\\\text{odd}}}^{\lfloor \frac{n-1}{2}\rfloor} |D_{n, k}^{<, c}|.\end{equation}

We will compute $|D_{n, k}^{<, c}|$ by exhibiting a bijection between $D_{n, k}^{<, c}$ and the set $\Gamma_{n, k, c} := \bigcup_{j = 1}^{k-1} Z_{j, c} \times A_{k-j, c},$
where 
$$Z_{j, c} = \{ z'_{[1, n-1]} : 0 \leq z'_i \leq j-1 \text{ for all $i$, and }  z'_{[n-2c, n-c-1]} = z'_{[n-c, n-1]} \text{ is primitive}  \},$$
and $A_{k-j, c}$ is the collection of multisets of  $\{ 2, \ldots, c\}$ of size $k-j$. 
Since $\card{A_{k-j, c}} =  {c + k - j -2 \choose k -j}$ and $\card{Z_{j, c}} = j^{n-2c -1} \psi_j(c)$ by construction, it will follow that 
$$|D_{n, k}^{<, c}| = |\Gamma_{n, k, c}| = \sum_{j = 1}^{k-1} |A_{k-j, c}|\,|Z_{j, c}| =
 \sum_{j =1}^{k-1} {c + k - j -2 \choose k -j} j^{n-2c-1}\psi_j(c),$$
which together with Equation~\eqref{eq:dnk} completes the proof.

Next we describe the bijection. Given $(\pi, E) \in D_{n, k}^{<, c}$, let $\zeta$ be the prefix defined by the $-k$-segmentation $E$ of $\hat{\pi}^{\star}$. We define a multiset $M$ by including each element $m$ for $1\le m\le c$ with multiplicity $|z_{n - 2c + m -1} - z_{n - c + m -1}|$.  Equivalently, for each $1 \leq m \leq c$, the multiplicity of $m$ in $M$ is the number of indices in $E$ equal to $\min\{\pi_{n-2c + m - 1}, \pi_{n-c+m-1}\}$.  By Definition~\ref{def:canonical} and Remark~\ref{rem:zxzy} together with the assumption that $E_{\pi} \not \subseteq E$, we have $z_{n - 2c} = z_{n - c}$, and so $M$ is a multiset of $\{ 2, \ldots, c\}$.  Let $j = k - |M|$, so that $M$ has size $k-j$.  Let $F= (f_0, f_1, \ldots, f_j)$ be the $-j$-segmentation of $\hat\pi^{\star}$  obtained by deleting from $E$ all of the indices  equal to $\min\{\pi_{n-2c + m - 1}, \pi_{n-c + m-1}\}$ for some $1 \leq m \leq c$ (these locations were recorded by the elements in $M$). Since $|\pi_{n-2c + i - 1} - \pi_{n-c+i -1}| = 1$ for all $2 \leq i \leq c$ by Lemma~\ref{lem:formcollapsed}, the prefix $\zeta'$ defined by $F$ satisfies $z'_{[n-2c, n-c-1]} = z'_{[n-c, n-1]}$.  Moreover, $z'_{[n-c, n-1]}$ is primitive by Lemma~\ref{prefixlemsd2}.  Hence, given $(\pi, E) \in D_{n, k}^{<, c}$, we obtain a pair $(z'_{[1, n-1]}, M) \in Z_{j, c} \times A_{k-j, c}$. See Example~\ref{ex:collapsedenum} for an illustration of this construction.

\medskip

In the rest of this proof we will show that the above map is a bijection by describing its inverse. Given $(z'_{[1, n-1]}, M) \in Z_{j, c} \times A_{k-j, c}$, we will recover the unique pair $(\pi, E) \in D_{n, k}^{<, c}$ that maps to it.  To obtain $\pi$, we define an intermediary permutation $\pi' \in \S_{n-c}$. Let $z' = z'_{[1, n-1]}(z'_{[n-c, n-1]})^{2n}((k{-}1)0)^\infty$, and let $\pi' = \Pat(z', \Sigma_{-k}, n-c)$, which is defined because $c$ being odd implies $(z'_{[n-c, n-1]})^\infty \neq ((k{-}1)0)^\infty$, and so the first $n-c$ shifts of $z'$ are distinct.  Moreover, note that 
 \begin{equation*}\label{eqn:2cc} z'_{[n-2c, \infty)} = (z'_{[n-c, n-1]})^{2n+2}((k{-}1)0)^\infty \gessalt   (z'_{[n-c, n-1]})^{2n+1}((k{-}1)0)^\infty =z'_{[n-c, \infty)},\end{equation*} 
 and so $\pi'_{n-2c} > \pi'_{n-c}$.  From $\pi'$, define $\pi \in \S_n$ to be the unique permutation such that
\begin{enumerate}[(i)]
\item  $\pi_{n-2c} = \pi_{n} + 1$ and $\pi_{n-c} = \pi_n - 1$,
\item  $\pi_{n-2c + j} - \pi_{n-c + j} = (-1)^j$ for all $1 \leq j < c$, and 
\item  for $1 \leq i, j \leq n-c$, we have $\pi_i < \pi_{j}$ if and only if $\pi'_{i} < \pi'_{j}$.
\setcounter{count}{\value{enumi}}
\end{enumerate}

In particular, $\pi$ satisfies $n-2c = x < y = n-c$ (with $x$ and $y$ as in Definition~\ref{def:pq}), and $\pi$ is collapsed by Lemma~\ref{lem:formcollapsed} together with condition~(ii).  
As a consequence of the above conditions, we also obtain 
\begin{enumerate}[(i)]
\setcounter{enumi}{\value{count}}
\item for $1 \leq i, h \leq n$, if $\pi_i < \pi_h$ then $z'_{i} \leq z'_h$;
\item for $1 \leq j \leq n-c$, we have $\pi_{n-c} < \pi_j$ if and only if $\pi'_{n-2c} \le \pi'_{j}$.
\end{enumerate}

Property (iv) is clear by construction. To see why (v) follows, suppose that $1 \leq j \leq n-c$ and $\pi_{n-c} < \pi_j$.  Then we must have $\pi_n=\pi_{n-c} + 1 < \pi_j$, since $j \neq n$, and thus $\pi_{n-2c} = \pi_{n-c} + 2  \le \pi_j$ using condition~(i).  By condition~(iii), we conclude that $\pi'_{n-2c} \leq \pi'_{j}$.  Conversely, if $\pi'_{n-2c} \leq \pi'_{j}$, then by condition~(iii), $\pi_{n-2c} = \pi_{n-c} + 2  \le \pi_j$, and so $\pi_{n-c}  < \pi_j$.

Let $F = (f_0, f_1, \ldots, f_j)$ be given by
\begin{equation} \label{eq:ft} f_t = \card{\{i : z'_{i}< t \text{ and }  1 \leq i \leq n \}} \end{equation} for $0 \leq t \leq j$. Next we show that $F$ is a $-j$-segmentation of $\hat{\pi}^{\star}$.

First of all, since $\pi$ is collapsed, we have $\pi_n \notin \{1, n\}$, and so conditions (b), (c), (d) and (e) of Definition~\ref{def:segmentation} trivially hold.  It remains to show that conditions (a) and (f) hold. 
\smallskip

To show that condition~(a) in Definition~\ref{def:segmentation} holds,
we will show that $\hat{\pi}^{\star}_{f_t + 1} \hat{\pi}^{\star}_{f_t + 2} \ldots \hat{\pi}^{\star}_{f_{t+1}}$ is decreasing for $0 \leq t \leq j-1$, disregarding the entry $\hat{\pi}^{\star}_{\pi_n} = \star$.  
Our goal is to prove that 
\begin{equation}\label{eq:claim1}
\text{for all $i, h < n$ such that $f_t < \pi_i < \pi_h \leq f_{t+1}$, we have $\hat{\pi}^{\star}_{\pi_i} > \hat{\pi}^{\star}_{\pi_h}$.} 
\end{equation}
The first step is to reduce this task to proving the following statement instead:
\begin{equation}\label{eq:claim2}
\text{for any $i, h < n-c$ such that $f_t < \pi_i < \pi_h \leq f_{t+1}$, we have $\hat{\pi}^{\star}_{\pi_i} > \hat{\pi}^{\star}_{\pi_h}$.}
\end{equation}

To show that~\eqref{eq:claim2} implies~\eqref{eq:claim1}, first note that the entries $\hat{\pi}^{\star}_{\pi_{n-2c}} \hat{\pi}^{\star}_{\pi_n} \hat{\pi}^{\star}_{\pi_{n-c}} = \hat{\pi}^{\star}_{\pi_n - 1} \hat{\pi}^{\star}_{\pi_n} \hat{\pi}^{\star}_{\pi_n+1} = \pi_{y+1}{\star}\pi_{x+1}$ do not contain an ascent since, taking $j=1$ in Lemma~~\ref{lem:formcollapsed}, we have $\pi_{x+1} - \pi_{y+1} = -1$. 
Again using Lemma~\ref{lem:formcollapsed}, we have $\hat{\pi}^{\star}_{\pi_{n-2c + j}} - \hat{\pi}^{\star}_{\pi_{n-c+j}} = \pi_{n-2c + j+1} - \pi_{n-c + j+1} = (-1)^{j+1}$ for any $1 \leq j \leq n-c-1$.  When $j$ is odd, this gives that $\pi_{n-2c + j} = \pi_{n-c+j} + 1$ and $\hat{\pi}^{\star}_{\pi_{n-2c + j}} =\hat{\pi}^{\star}_{\pi_{n-c + j}}-1$.  Similarly, when $j$ is even, this gives $\pi_{n-c + j} = \pi_{n-2c+j} + 1$ and $\hat{\pi}^{\star}_{\pi_{n-c + j}} = \hat{\pi}^{\star}_{\pi_{n-2c + j}}-1$.  Thus, for $1 \leq j \leq n-c$,  the entries $\hat{\pi}^{\star}_{\pi_{n-2c + j}}$ and $\hat{\pi}^{\star}_{\pi_{n-c+j}}$ are adjacent in $\hat{\pi}^{\star}$ and form a descent with the property that $|\hat{\pi}^{\star}_{\pi_{n-c+j}} - \hat{\pi}^{\star}_{\pi_{n-2c + j}}| = 1$.  Hence,~\eqref{eq:claim1} will follow once we prove~\eqref{eq:claim2}.

By Equation~\eqref{eq:ft}, $z'_{[1, n]}$ has $f_t$ letters smaller than $t$ and $f_{t+1}$ letters smaller than $t+1$. 
By property (iv), if $f_t < \pi_i$, then $z'_i$ is not one of the $f_t$ smallest letters, and if $\pi_i \le f_{t+1}$, then 
$z'_i$ is one of the $f_{t+1}$ smallest letters. It follows that
\begin{align}\label{eq:claim3}
f_t < \pi_i &\Rightarrow z'_i  \ge t, \\ 
\label{eq:claim4}
\pi_i \le f_{t+1} &\Rightarrow z'_i < t+1.
\end{align}
In particular, if $f_t < \pi_i \leq f_{t+1}$, then $z'_i  =t$. 

Suppose that $i, h < n-c$ are such that $f_t < \pi_i < \pi_h \leq f_{t+1}$.  By the above argument, $z'_i = z'_h = t$, and by condition~(iii), we have $\pi'_i < \pi'_h$.   Since $\Pat(z', \Sigma_{-k}, n-c) = \pi'$,  we have $z'_{[i, \infty)} \lessalt z'_{[h, \infty)}$, hence $z'_{[i+1, \infty)} \gessalt z'_{[h+1, \infty)}$ and $\pi'_{i+1} > \pi'_{h+1}$. Using condition~(iii) again, we conclude $\hat{\pi}_{\pi_i}^{\star} = \pi_{i+1} > \pi_{h+1} = \hat{\pi}_{\pi_h}^{\star}$.  This concludes the proof that condition~(a) in Definition~\ref{def:segmentation} holds.

Lastly, we must verify condition~(f), which says that $f_t \neq \pi_n$ for all $1 \leq t \leq k-1$.  Suppose for a contradiction that $f_t = \pi_n$.  By Equation~\eqref{eq:claim4} with $t$ in place of $t+1$, we have $z'_{n} < t$. And since $f_t < \pi_x = \pi_n + 1$, Equation~\eqref{eq:claim3} implies that $z'_{x} \ge t$. This contradicts the fact that $z'_n = z'_x$ by construction of $z'$. 
\smallskip

We conclude that $F$ is a $-j$-segmentation of $\hat{\pi}^{\star}$.  Moreover, by Equations~\eqref{eq:claim3} and~\eqref{eq:claim4}, the prefix defined by $F$ and $\pi$ is equal to $z'_{[1, n-1]}$.  

 Since $z'_{[n-2c, n-c-1]} = z'_{[n-c, n-1]}$, there is no index in $F$ at position $\min\{\pi_{n-2c + m-1}, \pi_{n-c+m-1}\}$ for each $m \in \{1, 2, \ldots, c\}$.  Let $E$ be the $-k$-segmentation of $\hat{\pi}^{\star}$ obtained by refining $F$ by adding, for each $m \in M$, the index $\min\{\pi_{n-2c -m-1}, \pi_{n-c + m -1}\}$ as many times as the multiplicity of $m$ in $M$.  Since $1 \notin M$, $E$ does not contain the index $\min\{\pi_{n - 2c}, \pi_{n-c}\} = \pi_y = \pi_n - 1$, and so $E_{\pi} \not \subseteq E$.  Thus, $(\pi, E) \in D_{n, k}^{<, c}$, completing the description of the inverse map.
\end{proof}

\begin{example}\label{ex:collapsedenum} To illustrate the bijection between $D_{n,k}^{<, c}$ and $\Gamma_{n,k,c}$ in the proof of Lemma~\ref{lem:collapsedenum}, let $\pi = 5173264$ and $E = (0,  1, 2, 5, 6, 6, 7)$, so $\hat\pi_E=|7|6|2{\star}1|4\|3|$. By Lemma~\ref{lem:formcollapsed}, $\pi$ is collapsed with $1=x < y=4$ and $c = n-y=3$, and $E$ is a valid $-6$-segmentation of $\hat{\pi}^{\star} = 762{\star}143$. Notice that $E_{\pi} = (0, 3, 5, 7)\not \subseteq E$, and so $(\pi,  E) \in D_{7, 6}^{<, 3}$.  The prefix defined by $E$ is $\zeta = 205213$, which has $|z_1 - z_4| = 0$, $|z_2 - z_5| = 1$ and $|z_3 - z_6| = 2$.  Therefore, the bijection produces the multiset $M = \multiset{2, 3, 3}$.  The $-3$-segmentation of $\hat{\pi}^{\star}$ obtained by deleting from $E$ the indices corresponding to bars in $\hat\pi_E$ between  $\pi_{m}$ and $\pi_{m + 3}$ for $1 \leq m \leq 3$  is $F = (0, 2, 5, 7)$, defining the prefix $\zeta' = 102102$.  Note that $z'_{[1, 3]} = z'_{[4, 6]}$ is primitive, and so $z'_{[1, 6]} \in Z_{3, 3}$.  We obtain the pair $(102102, \multiset{2, 3, 3}) \in Z_{3, 3} \times A_{4, 3}$. 

Conversely, starting from $(102102, \multiset{2, 3, 3}) \in Z_{3, 3} \times A_{4, 3}$, the construction in the proof of Lemma~\ref{lem:collapsedenum} gives $z' = (102)^{16}(50)^\infty$ and $\pi' = \Pat(z', \Sigma_{-6}, 3) = 213$.  The permutation satisfying (i), (ii), (iii) is $\pi = 5173264$.  Next, Equation~\eqref{eq:ft} with $z'_{[1, 7]} = 1021021$ yields $F = (0, 2, 5, 7)$.  Finally, we add additional indices to $F$, with multiplicty, by including the index $\min\{\pi_{m}, \pi_{m+3}\}$ for each $m \in M$. In this case, we include the index $\min\{ \pi_{2}, \pi_5\}  = 1$, and two copies of the index $\min\{\pi_3, \pi_6\} = 6$.  This gives $E = (0, 1, 2, 5, 6, 6, 7)$, resulting in the pair $(\pi, E) = (5173264, (0, 1, 2, 5, 6, 6, 7)) \in D_{7, 6}^{<, 3}$.  
\end{example}

\begin{lem}\label{lem:pnk} For $n \geq 3$ and $k \geq 2$, let 
\begin{multline*} p(n, k) = | \{ (\pi, E) : \pi \in \Al_n(\Sigma_{-k}) \text{ and }   \\ E \text{ is a valid $-k$-segmentation  of $\hat{\pi}^{\star}$ such that } E_{\pi} \subseteq E \}|.
\end{multline*}  
Then 
\begin{align*}\label{eq:pnk}
p(n, k)  & = a(n, k) + (k^2 - 2)k^{n-3} - 2 \sum_{j = 1}^{k-1} j^{n-3} 
- 2 \sum_{\substack{c = 1\\ odd}}^{\lfloor \frac{n-1}{2} \rfloor} \sum_{j =1}^{k-1}  {c + k - j -2 \choose k -j} j^{n-2c-1}\psi_j(c),
\end{align*}
where $a(n,k)$ as given by Equation~\eqref{ank}.
\end{lem}

\begin{proof}  
By Theorem~\ref{thm:numberalin}(c), the total number of pairs $(\pi, E)$, where $\pi \in \Al_n(\Sigma_{-k})$ and $E$ is a valid $-k$-segmentation of $\hat{\pi}^{\star}$, is equal to $I_n(\Sigma_{-k}) = a(n, k) + (k^2 - 2)k^{n-3}$.  To determine $p(n, k)$, we must subtract the number of pairs $(\pi, E)$ for which $E_{\pi} \not \subseteq E$.  

If $\pi$ is regular, there is only one $-\N(\pi)$-segmentation, and so if $E$ is a $-k$-segmentation of $\hat{\pi}^{\star}$, then trivially $E_{\pi} \subseteq E$.  If $\pi$ is cornered or collapsed, we must subtract the number of pairs $(\pi,  E)$ for which $E_{\pi} \not \subseteq E$.  We found the number of such pairs in Lemmas~\ref{lem:corneredenum} and~\ref{lem:collapsedenum}, respectively.  
We conclude
$$p(n, k) = I_n(\Sigma_{-k}) - |C_{n, k}| - |D_{n, k}|,$$
where $C_{n, k}$ and $D_{n, k}$ are defined by Equations~\eqref{eq:corneredenum} and~\eqref{eq:collapsedenum}. 
The formula for $p(n, k)$ now follows.    \end{proof}

\begin{thm}\label{thm:formula} For $n\ge3$ and $k\ge2$, let $b(n, k)$ be the number of permutations $\pi\in\S_n$ with $\N(\pi) = k$, that is, $b(n, k)=\card{\Al_n(\Sigma_{-k})\setminus\Al_n(\Sigma_{-(k-1)})} = \card{\{\pi \in \mathcal{S}_n : \N(\pi) = k\}}$. We have
$$p(n, k) = \sum_{j = 2}^{k} {n +k- j -1  \choose k-j} b(n, j).$$
Equivalently, 
$$b(n,k)=\sum_{j=2}^k (-1)^{k-j}\binom{n}{k-j}p(n,j),$$
where $p(n,j)$ is given by Lemma~\ref{lem:pnk}.
\end{thm}  

\begin{proof}
Recall that $p(n, k)$ is the number of pairs $(\pi, E)$ such that $\pi\in\Al_n(\Sigma_{-k})$ and $E$ is a valid $-k$-segmentation of $\hat{\pi}^{\star}$ satisfying $E_{\pi} \subseteq E$. We count these according to the value of $\N(\pi)$. 

Consider a pair $(\pi, E)$ as above with $\N(\pi)=j$, and note that $2 \leq j \leq k$.  Since $E_{\pi} \subseteq E$,  $E$ consists of the $j+1$ indices in $E_{\pi}$ plus $k-j$ additional indices, which are an arbitrary multiset of $\{0,1,\dots,n\}\setminus\{\pi_n\}$. It follows that there are $\binom{ n + k - j -1}{k - j}$ possible ways to choose the locations for these $k-j$ indices. Thus, for each $\pi\in\Al_n(\Sigma_{-k})$ with $\N(\pi) = j$, there are  ${n + k-j-1  \choose k- j}$ valid $-k$-segmentations  $E$ of $\hat{\pi}^{\star}$ such that $E_{\pi} \subseteq E$.
 Summing over all possible values of $j$, we obtain the first equality above. 

To derive the second statement from the first, note that
\begin{eqnarray*}
\sum_{k \ge 2} p(n, k)x^k &=&  \sum_{k \ge 2} \sum_{j = 2}^k {n + k - j - 1 \choose k-j} b(n, j) x^k \\
&=& \left(\sum_{k-j \ge 0} {k-j + n - 1 \choose k-j} x^{k-j}\right)\left(\sum_{j \ge 2} b(n, j)x^j \right) \\
&=&  \frac{1}{(1-x)^n} \sum_{j =2}^n b(n, j)x^j,
\end{eqnarray*}
and so $$\sum_{k=2}^n b(n, k) x^k = (1-x)^n \sum_{j \geq 2} p(n, j) x^j.$$ Extracting the coefficient of $x^k$ on both sides gives the stated formula.
\end{proof}

For completeness, we present the analogous result for positive shifts that was given in \cite[Thm.~5.13]{EliShift}, using segmentations to simplify the original counting argument.

\begin{thm}[\cite{EliShift}] \label{thm:elizalde} For $n \geq 3$ and $k \geq 2$, let $b'(n, k)$ be the number of permutations $\pi \in \S_n$ with $N(\pi) = k$, that is, $b'(n, k) = \card{\Al_{n}(\Sigma_k) \setminus \Al_{n}(\Sigma_{k-1})} = \card{\{\pi \in \mathcal{S}_n : N(\pi) = k\}}$.  We have
$$ \sum_{j = 2}^k {n +k- j -1 \choose k-j} b'(n, j) = I_n(\Sigma_k)= a(n, k) + (k-2)k^{n-2}.$$
Therefore, 
$$b'(n,k)=\sum_{j=2}^k \binom{n}{k-j} (-1)^{k-j} I_n(\Sigma_j),$$
with $I_n(\Sigma_j)$ given by Theorem~\ref{thm:numberalin}.
\end{thm}

\begin{proof} The number of pairs $(\pi, E)$ such that $\pi \in \Al_{n}(\Sigma_{k})$ and $E$ is a valid $k$-segmentation of $\hat{\pi}^{\star}$  equals $I_n(\Sigma_k) = a(n, k) + (k-2)k^{n-2}$ by Theorem~\ref{thm:numberalin}, verifying the second equality.  To show the first equality, we will count pairs $(\pi, E)$ according to the value of $N(\pi)$.

Consider a pair $(\pi, E)$ as above with $N(\pi)=j$, and note that $2 \leq j \leq k$.  Further recall that all $k$-segmentations are valid since $\|d\|=0$ for any $d$.  As discussed in the paragraph before Definition~\ref{def:ccr}, there is only one $N(\pi)$-segmentation of $\hat{\pi}^{\star}$, which we denote by $E'_\pi$, and so all $k$-segmentations $E$ trivially satisfy that $E'_{\pi} \subseteq E$.  
Therefore, $j+1$ of the indices in $E$ are those of $E'_{\pi}$, and the remaining $k-j$ indices are an arbitrary multiset of $\{0,1,\dots,n\}\setminus\{\pi_n\}$. Hence, for each $\pi$ with $N(\pi) = j$, there are $\binom{ n + k - j -1}{k - j}$ valid $k$-segmentations $E$ of $\hat{\pi}^{\star}$ with $E'_{\pi} \subseteq E$.  Summing over all possible values of $j$, we obtain the first equality above. 

The final equality follows as in the proof of Theorem~\ref{thm:formula}.  
\end{proof}

Theorem~\ref{thm:formula} and Lemma~\ref{lem:pnk} provide a formula for $\card{\Al_n(\Sigma_{-k})}=\sum_{j=2}^k b(n,j)$. The values of $b(n,k)$ for small $n$ and $k$ are given in Table~\ref{tab1}.  For comparison, we have also included the analogous values $b'(n,k)$ for the positive shift, which were obtained in~\cite{EliShift}\footnote{The entry $b'(8,4)=19476$ corrects a typo from \cite[Table 1]{EliShift}.}. 

\renewcommand{\arraystretch}{1.2}
\begin{table}
\centering
\setlength\tabcolsep{4.5pt}
$b'(n,k)=\card{\{\pi \in \S_n: N(\pi) = k\}}$ \\
{
\begin{tabular}{|c|r|r|r|r|r|r|r|r|} 
 \hline
 $n \setminus k$ & 2 & 3 & 4 & 5 & 6 & 7 & \ 8  \\ 
\hline
3&6 & & & & & & \\
\hline
4&18 &6 & & & & & \\
\hline
5& 48 & 66 &6 & & & & \\
\hline
6& 126 & 402 & 186 & 6 & & &  \\
\hline
7& 306 & 2028 & 2232 &468 & 6 & & \\
\hline
8& 738 & 8790 & 19476&  10212 & 1098 & 6 &  \\
\hline
9 & 1716& 35118& 137454& 144564& 41544& 2478& 6 \\
\hline
\end{tabular}} \medskip \\
$b(n, k) = \card{\{\pi \in \S_n: \N(\pi) = k\}}$ 
\\
{
\begin{tabular}{|c|r|r|r|r|r|r|r|r|} 
 \hline 
$n \setminus k$ & 2 & 3 & 4 & 5 & 6 &  7 & \ 8   \\  
\hline
3&6 & & & & & & \\
\hline
4&20 &4 & & & & & \\
\hline
5& 54 & 62 & 4& & & &  \\
\hline
6& 140 & 408 & 168 & 4 & & & \\
\hline
7 & 336 & 2084 & 2208 & 408 & 4 & &\\
\hline 
8 & 800 & 9152 & 19580 & 9820 & 964 &4 & \\
\hline
9 & 1842 & 36674 & 139760 & 142892 & 39514 & 2194 &4 \\
\hline
\end{tabular}}
\smallskip
\caption{Values of $b'(n,k)$ and $b(n, k)$, given by Theorem~\ref{thm:elizalde} and~\ref{thm:formula}, respectively. }
\label{tab1}
\end{table}

\begin{section}{Tent Bounds}\label{sec:tent}

Using Theorem~\ref{thm:intervals}, we obtain an improvement of the best-known bounds on the number of patterns realized by the tent map, the signed shift with signature $\Lambda = +-$. Finding an exact formula for $\card{\Al_n(\Sigma_{\Lambda})}$ is significantly more complicated than for the shift and negative shift.  The difficulty is that, if $\hat{\pi}^{\star}$ has a  $\Lambda$-segmentation $(e_0,e_1,e_2)$, then we may choose $e_1$ to be on either side of the peak of $\hat{\pi}^{\star}$, giving two possible $\Lambda$-segmentations.  However, one or both of the $\Lambda$-segmentations may be invalid, and in the latter case, Theorem~\ref{characterization} implies that $\pi  \notin \Al(\Sigma_\Lambda)$.

The first few values of the sequence $\card{\Al_n(\Sigma_{\Lambda})}$ for $n\ge1$, obtained by brute-force computation, are $1,2,5,12,31,75,178,414,949,2137,4767,10525,23042,\dots$.

By Theorem~\ref{thm:numberalin},
\begin{equation*}I_n(\Sigma_{\Lambda}) = a(n, 2) + 2^{n-2} = \sum_{i = 1}^{n-1} 2^{n-i-1}\psi_2(i) + 2^{n-2}, \end{equation*}
where $\psi_{2}(m)$ is the number of primitive binary words of length $m$.  Let $\overline{\psi}_2(m)$ denote the number of those with an odd number of ones. It is shown in \cite[Thm~3.4]{KACyclicSigned} that 
\begin{equation}\label{eq:oddprimitive} \overline{\psi}_2(m) = \frac{1}{2} \sum_{\substack{d |m \\ d \text{ odd}}} \mu(d) 2^{m/d}. \end{equation}

The lower bound on $\card{\Al_n(\Sigma_{\Lambda})}$ in the next theorem was first found in \cite[Thm.~5.5]{KAchar}.  We present a complete proof using Theorem~\ref{characterization}.  

\begin{thm} \label{thm:tentbounds} For the tent map, the signed shift with signature $\Lambda = +-$ and $n \ge 3$, we have
$$\frac{1}{2}I_n(\Sigma_{\Lambda}) \leq \card{ \Al_n(\Sigma_\Lambda)}  \leq \frac{1}{2}I_n(\Sigma_{\Lambda}) +  \sum_{c = 1}^{\lfloor \frac{n-1}{2} \rfloor} 2^{n - 2c - 1} \overline{\psi}_2(c).$$
\end{thm}

\begin{proof} Recall that $I_n(\Sigma_{\Lambda})$ is the number of pairs $(\pi, E)$ where $\pi \in \Al_n(\Sigma_{\Lambda})$ and $E$ is a valid $\Lambda$-segmentation of $\hat{\pi}^{\star}$. Furthermore, for $\pi \in \Al_n(\Sigma_{\Lambda})$, conditions (b)-(e) in Definition~\ref{def:segmentation} do not apply for since the only possibility is (b).  In this case, $\hat{\pi}^{\star}$ would be required to have $\hat{\pi}^{\star}_1\hat{\pi}^{\star}_2 = {\star}1$ and $\hat{\pi}^{\star}_2 \hat{\pi}^{\star}_3 \ldots \hat{\pi}^{\star}_n$ is decreasing, which is impossible for $n \geq 3$.  Therefore, if $\pi \in \Al_n(\Sigma_{\Lambda})$, then $\hat{\pi}^{\star}$ must be unimodal, and so it has two $\Lambda$-segmentations $E = (e_0, e_1, e_2)$, determined by whether $e_1$ is to the left or right of the peak. Thus, for each $\pi \in \Al_n(\Sigma_{\Lambda})$, the marked cycle $\hat\pi$ has at most two valid $\Lambda$-segmentations. It follows that 
$$I_n(\Sigma_\Lambda) \leq 2 \card{\Al_n(\Sigma_\Lambda)},$$
giving the lower bound.

To obtain an upper bound, let $\Theta_n$ be the number of pairs $(\pi, E)$, where $\pi \in \S_n$ and $E$ is an invalid $\Lambda$-segmentation of $\hat{\pi}^{\star}$. The total number of pairs $(\pi, E)$, where now $E$ is any (valid or invalid) $\Lambda$-segmentation of $\hat\pi$, is then equal to $I_n(\Sigma_\Lambda) + \Theta_n$. Since for each $\pi \in \Al_n(\Sigma_\Lambda)$, the marked cycle $\hat{\pi}^{\star}$ has exactly two $\Lambda$-segmentations, we obtain
$$ 2 \card{\Al_n(\Sigma_\Lambda)} \leq I_n(\Sigma_\Lambda) + \Theta_n.$$

To find $\Theta_n$, we determine the number of invalid $\Lambda$-segmentations according to the prefixes that they define.  By Lemmas~\ref{prefixlemsd2} and~\ref{lem:doubleback}, the prefix defined by an invalid $\Lambda$-segmentation is of the form $\zeta = upp$ where $\|p\| = |\{i: x \leq i <n , z_{i} = 1 \}|$ is odd and $p$ is primitive, or $\zeta = u'qq$ where $\|q\| = |\{i: y \leq i < n , z_i = 1\}|$ is odd and $q$ is primitive.  Counting these prefixes according to $|p|$ and $|q|$, respectively, denoting this length by $c$,  we get
$$\Theta_n = 2 \sum_{c = 1}^{\lfloor \frac{n-1}{2} \rfloor} 2^{n - 2c - 1} \overline{\psi}_2(c),$$
where $\overline{\psi}$ is given by Equation~\eqref{eq:oddprimitive}.  
\end{proof}

The upper and lower bounds in Theorem~\ref{thm:tentbounds} are not tight for $n > 4$. Indeed, having a $\Lambda$-segmentation of $\hat{\pi}^{\star}$ is not enough to guarantee that $\pi$ is an allowed pattern of the tent map.  The following is an example of a permutation that is counted by the upper bound, yet is not an allowed pattern.  

\begin{example} Let $\pi = 51423$, and so $\hat{\pi}^{\star} = 43{\star}21$.  The two possible $\Lambda$-segmentations of $\hat{\pi}^{\star}$ are $(0, 0, 5)$ and $(0, 1, 5)$, defining the prefixes $\zeta = 1011$ and $\zeta= 1111$, respectively.  However, both $\Lambda$-segmentations are invalid because, in either case, we have $q^2 = (1)^2 = p$.  It follows that $\pi \notin \Al_5(\Sigma_\Lambda)$.  
\end{example}  

The reason why the lower bound is also not tight is that it is possible for an allowed permutation $\pi$ to appear in exactly one allowed interval. This happens when $\pi$ has one valid and one invalid $\Lambda$-segmentation.  

\begin{example}  Let $\pi = 32514$, and so $\hat{\pi}^{\star} = 452{\star}1$.   One $\Lambda$-segmentation of $\hat{\pi}^{\star}$ is given by $E = (0, 1, 5)$, and so $\hat{\pi}^{\star}_{E} = |4|52{\star}1|$, and defines the prefix $\zeta = 1110$.  Since $p = 10$ and $q = 1110$, this segmentation is valid.  The other $\Lambda$-segmentation of $\hat{\pi}^{\star}$ is given by $E = (0, 2, 5)$, and so $\hat{\pi}^{\star}_{E} =|45|2{\star}1|$, and defines the prefix $\zeta = 1010$.  Since $p^2 = (10)^2 = q$, this segmentation is invalid.  Therefore, $\pi$ is an allowed pattern appearing in exactly one allowed interval.  
\end{example}  

We now give some intuition showing why finding an explicit formula for $\card{\Al_n(\Lambda)}$ is considerably more challenging than for $\card{\Al_n(\Sigma_{k})}$ and $\card{\Al_n(\Sigma_{-k})}$.  Suppose that $\pi \in \S_n$ is such that $\hat{\pi}^{\star}$ has a $\Lambda$-segmentation. Let $\ell$ be the index such that $\pi_{\ell} = n$ except if $\pi_1 = n$, in which case take $\pi_{\ell} = n-1$.  Using Definition~\ref{def:segmentation}, the two $\Lambda$-segmentations of $\hat{\pi}^{\star}$ define prefixes $\zeta$ and $\zeta'$ such that $z_{\ell - 1} \neq z'_{\ell-1}$ and $z_{i} = z'_{i}$ for all $1 \leq i \leq n-1$, $i \neq \ell - 1$. Therefore, to determine the number of $\pi \in \Al_n(\Lambda)$ with exactly one valid $\Lambda$-segmentation, one would have to be able to count such prefixes while tracking the location of each of $\ell$, $x$ and $y$.

The following conjecture remains open for $k \geq 3$. 

\begin{conj}[\cite{KAchar}] For any $n \geq 3$ and any signature $\sigma = \sigma_0 \sigma_1 \ldots \sigma_{k-1}\in\{+,-\}^k$ with $\sigma \neq -^k$, we have $$\card{\Al_n(\Sigma_{\sigma})} \leq \card{\Al_n(\Sigma_{k})} \leq \card{\Al_n(\Sigma_{-k})}.$$ \end{conj}

Let us show that the above conjecture holds for $k = 2$.  As a corollary to Theorem~\ref{thm:formula}, we have
$$\card{\Al_n(\Sigma_{-2})} = a(n, 2) + 2^{n-2} - 2,$$
which was also obtained previously in \cite[Thm.~5.11]{KAchar}.
Further, as a corollary to \cite[Thm.~5.13]{EliShift}, stated here as Theorem~\ref{thm:elizalde}, we have 
$$\card{\Al_n(\Sigma_{2})} = a(n, 2).$$
Using the symmetry of allowed patterns, discussed immediately after Definition~\ref{def:segmentation}, we have
$$\card{\Al_{n}(\Sigma_{+-})} = \card{\Al_{n}(\Sigma_{-+})}.$$
Together with Theorem~\ref{thm:tentbounds}, we have now verified the conjecture for $k=2$.  
\end{section}

\section{Topological Entropy of Signed Shifts}\label{sec:entropy}
 
Topological entropy is an important measure of the complexity of a dynamical system.  For a given continuous map $f$, the topological entropy is equal to the logarithm of the exponential growth rate of the number of distinguishable orbits. When $f$ is piecewise monotone and not necessarily continuous, the topological entropy may be defined in terms of partitions.  

Let $f$ be a piecewise monotone map of a bounded real interval $I$ and let $\mathcal{P}$ be the a partition of $I$ into the finitely many intervals for which $f$ is strictly monotone.  Now let $\mathcal{P}_n$ be the partition of $I$ into the non-empty sets of the form $A_{i_1} \cap f^{−1}(A_{i_2}) \cap \dots \cap f^{-(n-1)}(A_{i_n})$ for $A_{i_1}, A_{i_2}, \ldots, A_{i_n} \in \mathcal{P}$. The topological entropy of $f$ is taken to be 
$$\lim_{n \rightarrow \infty} \frac{1}{n} \log(\card{\mathcal{P}_n}),$$
which coincides with the definition of topological entropy in terms of covering sets \cite{MSEntropy}.

It is shown in~\cite{BKP} that the topological entropy of a piecewise monotone map $f$ on a real interval $I$ equals the {\em permutation topological entropy} of $f$, which is given by
$$\lim_{n \rightarrow \infty} \frac{ \log(\card{\Al_{n}(f)})}{n-1}.$$
The following consequence of Theorem~\ref{thm:numberalin} provides a combinatorial way to obtain the topological entropy of the signed shift, which is already known, but is typically computed using  other methods~\cite{MSEntropy}.

\begin{cor}  For any $\sigma\in\{+,-\}^k$, the topological entropy of $M_{\sigma}$ is $\log(k)$.   
\end{cor}

\begin{proof} Recall that $\card{\Al_{n}(M_{\sigma})} = \card{\Al_{n}(\Sigma_{\sigma})}$. 
For $\pi\in\Al_n(\Sigma_\sigma)$, the number of distinct prefixes defined by a $\sigma$-segmentation of $\hat\pi$ is at most  $ {n + k -2 \choose k-1}$ because a $-k$-segmentation $E$ is determined by a multiset of $k-1$ indices between $0$ and $n$ excluding $\pi_n$. Thus, each permutation can appear in at most ${n + k -2 \choose k-1}$ allowed intervals.  It follows that \begin{equation}\label{eq:entineq} \frac{I_n(\Sigma_{\sigma})}{{n +k -2 \choose k-1}} \leq \card{\Al_n(\Sigma_{\sigma})} \leq I_n(\Sigma_\sigma).\end{equation}
By Theorem~\ref{thm:numberalin}, $I_n(\Sigma_{\sigma})  = \sum_{i = 1}^{n-1} k^{n-i-1} \psi_k(i) + k^{n-1} + O(k^{n-2})$ in all cases.  Since $\psi_k(i)$ is the number of primitive $k$-ary words of length $i$, we have $\sum_{i = 1}^{n-1} k^{n-i-1} \psi_k(i)\sim(n-1) k^{n-1}$, and so
$I_n(\Sigma_{\sigma})\sim nk^{n-1}$, where we write $f(n)\sim g(n)$ to mean $\lim_{n\to\infty} f(n)/g(n)=1$. To complete the proof, we use Equation~\eqref{eq:entineq} together with the fact that 
$$\lim_{n \rightarrow \infty} \frac{\log(n k^{n-1})}{n-1} =  \log(k) \quad \text{and} \quad 
\lim_{n \rightarrow \infty} \frac{\log(n k^{n-1}) - \log({n + k -2 \choose k-1})}{n-1} =  \log(k).$$
\end{proof}

\bibliographystyle{amsplain}
\bibliography{SignedShift}

\providecommand{\bysame}{\leavevmode\hbox to3em{\hrulefill}\thinspace}
\providecommand{\MR}{\relax\ifhmode\unskip\space\fi MR }
\providecommand{\MRhref}[2]{%
  \href{http://www.ams.org/mathscinet-getitem?mr=#1}{#2}
}
\providecommand{\href}[2]{#2}
\begin{thebibliography}{10}

\bibitem{AmigoSigned}
Jos{\'e}~Mar{\'{\i}}a Amig{\'o}, \emph{The ordinal structure of the signed
  shift transformations}, International Journal of Bifurucation and Chaos
  \textbf{19} (2009), 3311--3327.

\bibitem{Amigobook}
\bysame, \emph{{P}ermutation {C}omplexity in {D}ynamical {S}ystems},
  Springer-Verlag, Berlin, 2010, Springer Series in Synergetics.

\bibitem{KAchar}
Kassie Archer, \emph{Characterization of the allowed patterns of signed
  shifts}, Discrete Applied Math \textbf{217} (2017), 97--109.

\bibitem{KACyclicSigned}
Kassie Archer and Sergi Elizalde, \emph{Cyclic permutations realized by signed
  shifts}, J. Comb. \textbf{5} (2014), no.~1, 1--30.

\bibitem{KASLsymtent}
Kassie Archer and Scott~M. LaLonde, \emph{Allowed patterns of symmetric tent
  maps via commuter functions}, SIAM Journal on Discrete Mathematics
  \textbf{31} (2017), no.~1, 317--334.

\bibitem{BKP}
Christoph Bandt, Gerhard Keller, and Bernd Pompe, \emph{Entropy of interval
  maps via permutations}, Nonlinearity \textbf{15} (2002), 1595--1602.

\bibitem{ECWSnegbeta}
{\'E}milie Charlier and Wolfgang Steiner, \emph{Permutations and negative
  beta-shifts}, Preprint, \texttt{arXiv:1702.00652}. To appear in the
  International Journal of Foundations of Computer Science.

\bibitem{EliShift}
Sergi Elizalde, \emph{The {Number} of {Permutations} {Realized} {By} a
  {Shift}}, SIAM Journal on Discrete Mathematics \textbf{23} (2009), no.~2,
  765--786.

\bibitem{EliBeta}
\bysame, \emph{Permutations and {$\beta$}-shifts}, Journal of Combinatorial
  Theory Series A \textbf{118} (2011), no.~8, 2474--2497.

\bibitem{EliLiu}
Sergi Elizalde and Yangyang Liu, \emph{On basic forbidden patterns of
  functions}, Discrete Applied Math \textbf{159} (2011), 1207, 1216.

\bibitem{EMnegbeta}
Sergi Elizalde and Katherine Moore, \emph{Patterns of negative shifts and
  beta-shifts}, Preprint, \texttt{arXiv:1512.04479}.

\bibitem{MSEntropy}
Micha{\l} Misiurewicz and Wieslaw Szlenk, \emph{Entropy of piecewise monotone
  mappings}, Studia Math \textbf{67} (1980), 45--63.

\end{thebibliography}

\end{document}